\documentclass[reqno,11pt]{amsart}
\usepackage{amsmath,amssymb,latexsym,soul,cite,mathrsfs}
\usepackage{dsfont}
\usepackage{color,enumitem,graphicx}
\usepackage{xcolor}
\usepackage[english]{babel}
\usepackage[T1]{fontenc}  
\usepackage[utf8]{inputenc}
\usepackage{tikz}
\makeatletter
\tikzset{every picture/.append style={execute at begin picture=\fussy}}
\makeatother

\usepackage[colorlinks=true,urlcolor=blue,
citecolor=red,linkcolor=blue,linktocpage,pdfpagelabels,
bookmarksnumbered,bookmarksopen]{hyperref}

\usepackage[left=2.6cm,right=2.6cm,top=2.9cm,bottom=2.9cm]{geometry}
\usepackage[hyperpageref]{backref}

\usepackage{academicons}
\definecolor{orcidlogocol}{HTML}{A6CE39}

\makeatletter
\def\namedlabel#1#2{\begingroup
    #2%
    \def\@currentlabel{#2}%
    \phantomsection\label {#1}\endgroup
}
\makeatother

\newtheorem{theorem}{Theorem}[section]
\newtheorem{lemma}[theorem]{Lemma}
\newtheorem{corollary}[theorem]{Corollary}
\newtheorem{example}{Example}
\newtheorem{proposition}[theorem]{Proposition}
\newtheorem{remark}[theorem]{Remark}
\newtheorem{definition}[theorem]{Definition}

\newtheoremstyle{theorem}
{}                
{}                
{\slshape}        
{}                
{\bfseries}       
{'}               
{ }               
{}                
\theoremstyle{theorem}

\newcommand{\Ng}{\mathscr{N}_{\sigma}}
\newcommand{\Ngp}{\mathscr{N}_{\sigma_{+}}}
\DeclareMathOperator*{\esssup}{ess\,sup}
\DeclareMathOperator*{\essinf}{ess\,inf}

\newcommand{\orcid}[1]{\href{https://orcid.org/#1}{\textcolor[HTML]{A6CE39}{\aiOrcid}}}

\definecolor{lime}{HTML}{A6CE39}
\DeclareRobustCommand{\orcidicon}{
	\begin{tikzpicture}
	\draw[lime, fill=lime] (0,0) 
	circle [radius=0.16] 
	node[white] {{\fontfamily{qag}\selectfont \tiny ID}};
	\draw[white, fill=white] (-0.0625,0.095) 
	circle [radius=0.007];
	\end{tikzpicture}
 \hspace{-2mm}
 }
 
\foreach \letter in {A, ..., Z}{\expandafter\xdef\csname orcid\letter\endcsname{\noexpand\href{https://orcid.org/\csname orcidauthor\letter\endcsname}
			{\noexpand\orcidicon}}
}

\title[Steady states of FitzHugh-Nagumo-type systems with sign-changing coefficients]{Steady states of FitzHugh-Nagumo-type systems with sign-changing coefficients}

\author[J.M. do \'O]{Jo\~ao Marcos do \'O \orcidJ{}}
\author[E. Shamarova]{Evelina Shamarova \orcidE{}}
\author[ Victor V. Silva]{Victor V. Silva\orcidV{} }

\address[J.M. do \'O]{Department of Mathematics,
	Federal University of Para\'{\i}ba
	\newline\indent
	58051-900, Jo\~ao Pessoa-PB, Brazil}
\email{\href{mailto:jmbo@mat.ufpb.br}{jmbo@mat.ufpb.br}}

\address[E. Shamarova]{Department of Mathematics,
	Federal University of Para\'{\i}ba
	\newline\indent
	58051-900, Jo\~ao Pessoa-PB, Brazil}
\email{\href{mailto:evelina.shamarova@academico.ufpb.br}{evelina.shamarova@academico.ufpb.br}}

\address[ Victor V. Silva]{Department of Mathematics,
		Federal University of Pernambuco
	\newline\indent
	 50670-901, Recife, Brazil}
\email{\href{mailto:victor.franca@ufpe.br}{victor.franca@ufpe.br}}

\subjclass[2020]{35J60, 35B09, 35J30}
\keywords{FitzHugh-Nagumo systems,  Elliptic systems, Variational methods, Nonlinear equations}

\begin{document}
\begin{abstract}
We establish existence and multiplicity results for steady-state solutions of spatially heterogeneous FitzHugh--Nagumo-type systems, extending the existing theory from constant to variable coefficients that may change sign. Specifically, we study the system
\begin{equation}\label {P}\tag{$S$}
		\left\lbrace
		\begin{aligned}
			-\Delta u +a(x)v  &=f(x,u) &  & \text { in } \mathds{R}^N, \\
			-\Delta v+b(x)v &=c(x)  u & & \text { in } \mathds{R}^N,
		\end{aligned}\right.
    \end{equation}
where $N \geqslant 3$, the coefficients $a,b,c : \mathbb{R}^N \to \mathbb{R}$ are $L^\infty_{\mathrm{loc}}$-functions bounded from below, and $f:\mathbb{R}^N \times \mathbb{R} \to \mathbb{R}$ is a Carath\'eodory function with subcritical growth. For assumptions permitting sign changes and non-coercivity of the coefficients, we prove the existence of a mountain pass solution. In the case where $a,b,c$ do not change sign, still allowing non-coercive behavior, we additionally establish the existence of componentwise positive and negative solutions. 
\end{abstract}

	\maketitle

	\begin{center}
		\footnotesize
		\tableofcontents
   \end{center}

\section{Introduction}
\label{intro-1111}

\subsection{Motivation}
The FitzHugh-Nagumo model, originally introduced by FitzHugh \cite{fitzhugh1961impulses} and further developed by Nagumo et al. \cite{Nagumo19622061},
provides the canonical framework for studying excitable systems in neuroscience, cardiology, and pattern formation. While the classical model employs constant coefficients, real excitable media exhibit substantial spatial heterogeneity.
In neural tissue, variations in ion channel density and neuromodulator concentrations cause spatially varying coupling strengths that can even switch 
from excitatory to 
inhibitory \cite{Lenk2010InfluenceOS, zbMATH07966032}. In cardiac tissue, ischemic zones locally alter electrical properties, fundamentally affecting signal propagation \cite{zbMATH07966032}.
Furthermore, external influences like noise can introduce additional spatial heterogeneity \cite{zbMATH07966032}. 
Despite this biological reality, the mathematical theory for steady states has been largely confined to spatially homogeneous settings.

The classical FitzHugh-Nagumo system reads  \cite{MR4588070}
\[
\begin{cases}
u_t - \Delta u + v = g(u) & \text{in } \mathds{R}^N, \\
v_t - \Delta v + \gamma v = \beta u  & \text{in } \mathds{R}^N,
\end{cases}
\]
where \(\gamma, \beta, \delta > 0\) are constant coupling strengths and the nonlinearity \(g(u) = u(1 - u)(u - \alpha)\), with \(\alpha \in (0, \frac12)\), 
represents the activator kinetics. 
This model reduces complex action potential dynamics to a system of two coupled differential equations that capture 
the interplay between an activator variable $u$, e.g., membrane potential in neurons, and an inhibitor variable $v$. 
This simplification has yielded profound insights into neural excitability, 
cardiac rhythms, and pattern formation in reaction-diffusion systems.
Steady-state solutions to this system are essential to understanding resting states, 
excitability thresholds, and spatial patterns, such as Turing patterns.

However, all existing variational studies, e.g.,  \cite{MR3552796, MR4358261, MR3965640, MR3197242}, 
assume constant positive coefficients. We remove this restriction entirely, allowing $a(x), b(x), c(x)$ to be merely 
$L^\infty_{\rm loc}$-functions that may take both positive and negative values, with no coercivity or decay requirements.
More specifically, we extend the existing results to system \eqref{P}, for which we prove the existence of solutions for sign-changing coefficients, and also,
the existence of componentwise positive and negative solutions for a slightly more restrictive  class of coefficients;
however, in both cases under quite general assumptions on the nonlinearity $f$.
We believe that our extension enhances the applicability  of the model to complex, heterogeneous systems encountered in practice.

Mathematically, incorporating heterogeneity, alternation of signs, and lack of coercivity in the coefficients 
introduces substantial challenges, as traditional analytical methods for establishing the existence and 
multiplicity of steady-state solutions fail. Overcoming these obstacles requires new approaches, 
such as development of an adequate  variational framework and recovering compactness in a non-coercive setting.
We believe that the techniques developed in this work might be suitable for other 
PDEs with spatially variable, sign-changing, and non-coercive coefficients.

\subsection{  \texorpdfstring{Hypotheses on the coefficients 
and classes of functions $\Upsilon_{s}(\mathds{R}^{N})$
 and $\mathcal T(\mathds{R}^{N})$} \ \   }
 \label{s12}
\begin{definition}[Class $\Upsilon_{s}(\mathds{R}^{N})$]
\label{Upsilon}
\rm
    For a fixed $s\in [2,2^{*})$,  we define the class 
    $\Upsilon_{s} \left(\mathds{R}^{N}\right)$   
    of real-valued $L^\infty_{\mathrm{loc}}(\mathds{R}^{N})$-functions $\sigma$
essentially bounded from below and 
     satisfying conditions \eqref{primeiroAutovalor} and \eqref{coecivo}:
    \begin{equation}\label {primeiroAutovalor}
\lambda_{1}(\sigma):= \inf \left\lbrace \int_{\mathds{R}^{N}}\left(|\nabla u|^2+\sigma(x)|u|^2\right) \mathrm{d}x:\|u\|_{\scriptscriptstyle L_2}=1 \text{ and } u\in {\rm C}_{0}^{\infty}(\mathds{R}^{N}) \right\rbrace>0,
\tag{$\lambda_1$}
    \end{equation}
    and
     \begin{equation}\label {coecivo}
          \lim _{n \rightarrow \infty} \nu_s\left(\sigma, B_{R}(x_{n})\right)=\infty,
          \tag{$\nu_s$}
      \end{equation}
    where $\left\{x_n\right\} \subset \mathds{R}^{N}$, $\lim_{n\to\infty} x_n = \infty$, $R>0$, and 
    \begin{equation}\label {nus}
         \nu_{s}(\sigma,\Omega):=\inf \left\{\frac{\int_{\Omega}\left(|\nabla u|^2+\sigma(x)|u|^2\right) \, \mathrm{d}x}{\|u\|^{2}_{\scriptscriptstyle L_{s}(\Omega)}}: u \in H_0^{1}(\Omega) \setminus  \{0\} \right\},
    \end{equation}
    where $\Omega\subset \mathds{R}^N$ is a non-empty open set and $\nu_s(\sigma,\varnothing) = \infty$.
\end{definition}
\begin{remark}
\label{usp-s-1111}
\rm
By \cite[Lemma 2.2]{MR1782990}, if condition \eqref{coecivo} holds for $s=2$, then
it holds for all $s\in (2,2^*)$. Furthermore, if \eqref{coecivo} holds for some $s\in (2,2^*)$, then it holds
for all $s\in (2,2^*)$. This means that
\begin{align*}
\Upsilon_2(\mathds{R}^N) \subset \Upsilon_s(\mathds{R}^N) \;\; \forall \; s\in (2,2^*) \qquad \text{and} \qquad
\Upsilon_{s} = \Upsilon_{t} \quad \forall \; s,t\in (2,2^*).
\end{align*}
\end{remark}    
\begin{definition}[Class $\mathcal T(\mathds{R}^{N})$]
\label {T-1111}
\rm
    We define the class 
    $\mathcal T \left(\mathds{R}^{N}\right)$  of real-valued
    $L^\infty_{\mathrm{loc}}(\mathds{R}^{N})$-functions 
   essentially bounded from below  and  satisfying condition \eqref{primeiroAutovalor}.
  \end{definition}   
 Given a function $\sigma \in \mathcal T(\mathds{R}^{N})$, define the linear vector space
      \begin{equation*}
         E_{\sigma}:=\left\lbrace  u\in D^{1,2}(\mathds{R}^{N}):  \int_{\mathds{R}^{N}} \sigma_{+}(x) |u|^{2} \, \mathrm{d}x<+\infty \right\rbrace
    \end{equation*}
    which becomes a normed space with respect to the norm
     \begin{align*}
       \|u\|_{\sigma_+}=  \Big(\int_{\mathds{R}^{N}} |\nabla u|^2 +\sigma_+(x)|u|^2\,  \mathrm{d}x\Big)^\frac12,
\end{align*}
where $\sigma_+ = \max\{0,\sigma\}$.
From the results of \cite[Section 3]{MR924157} it follows that
 $(E_\sigma, \|\cdot\|_{\sigma_+})$ is a Hilbert space.
  The space $E_{\sigma}$ is described in more detail in Section \ref{Properties of space}.

  Remark that a function $\sigma\in\Upsilon_s(\mathds{R}^N)$, $s\in [2,2^*)$, is allowed to take negative values.
In Subsection \ref{Important examples}, we give an example  of a function (see Example \ref{Example_Sigma})
  taking negative values and satisfying \eqref{primeiroAutovalor} and  \eqref {coecivo}.
  However, if $\essinf \sigma(x) \geqslant \varepsilon$ for some $\varepsilon > 0$,  then $\lambda_1(\sigma) \geqslant \varepsilon $, which means that condition \eqref{primeiroAutovalor} is automatically satisfied.
   One can show that condition \eqref{coecivo} generalizes the coercivity condition ($\lim_{x\to\infty}\sigma(x) = \infty$) used by Rabinowitz in~\cite{rabinowitz1992class}. 
   Similar conditions appear in other works, see, e.g.,~\cite{MR1724275} and~\cite{MR1782990}.
   
  It will be convenient to formulate assumption \ref{f0a}  on the coefficients $a$ and $b$, 
  which will be used \textit{selectively} in the lemmas, propositions, and theorems throughout the paper.
  \begin{enumerate}
    \item[\namedlabel{f0a}{($h_a$)}]   
     $ a \in \mathcal T (\mathds{R}^{N}) $ or $a \geqslant 0$ (a.e.) is measurable with $a \ne 0$ a.e.  Moreover, $E_a = E_b$ and the norms
 $\|\cdot\|_{a_+}$ and $\|\cdot\|_{b_+}$ generate the same topology
 on $E_b$. 
 \end{enumerate}
 Condition \ref{f0ab} will only be used in Theorem \ref{teo}.
 \begin{enumerate}
  \item[\namedlabel{f0ab}{($h_{e}$)}] 
  There exists a measurable function $\vartheta: \mathbb{R}^N\to [0,+\infty)$ and 
  a constant $C_\vartheta>0$ such that
   such that $0\leqslant e: = b-2\beta^{\frac12}a  \leqslant C_\vartheta (1+\vartheta^\frac1{\alpha})$
   for some $\alpha>\max\{2,N/2\}$,
  and the norms $\|\cdot\|_\vartheta$ and $\|\cdot\|_{b_+}$ generate the same topology
 on $E_b$. 
  \end{enumerate}
    \subsection{Representative case of the main result}
    The following theorem is a representative case of our main result. 
    More specifically, let us consider system \eqref{P}
    with the classical nonlinear part $f(x,u) = |u|^{p-1} u$, which
    is a particular case of our more general class of nonlinearities
    described by hypotheses \ref{f1}-\ref{f5}.   
     \begin{theorem}
 \label {t-main}
  Let $ b \in \Upsilon_2(\mathds{R}^{N})$
and  let {\rm \ref{f0a}} hold.  Assume $p \in (1, 2^*-1)$.
 Then, the system
    \begin{equation}
    \label{rep-ex-1111}
		\left\lbrace
		\begin{aligned}
			-\Delta u +a(x)v  &= |u|^{p-1}u &  & \text { in } \mathds{R}^N, \\
			-\Delta v+b(x)v &=a(x)u & & \text { in } \mathds{R}^N
		\end{aligned}\right.
	\end{equation}
    has a mountain pass solution. If, moreover, 
    {\rm \ref{f0ab}} holds and $a>0$ a.e., then system \eqref{rep-ex-1111}
    has three nontrivial weak solutions $(u_1, v_1), (u_2, v_2)$ and $(u_3, v_3)$, such that      
    \begin{align*}
     u_1> 0, \; v_1>0;  \quad u_2< 0, \; v_2<0; \quad  
     u_{3} \; \text{is sign-changing}.
\end{align*}
    \end{theorem}

\subsection{Hypotheses on the nonlinearity $f$}

As we mentioned before, Theorem \ref{t-main} is a particular
case of a more general result which is the subject of  Theorems \ref{teo1} and \ref{teo}.
More specifically, we assume that the nonlinear term $f$ satisfies
the following below hypotheses \ref{f1}--\ref{f4} in the case of Theorem \ref{teo1}
and \ref{f1}--\ref{f5} in the case of Theorem \ref{teo}.
Let $\alpha>\max\{2,N/2\}$. We define 
\begin{align}
\label{PalfN}
\mathcal P_{\alpha,N}: =
\Big(1, 1+\frac2N\Big] \bigcup
\Big[\frac{N-\frac4\alpha}{N-2}, \frac{N-\frac4\alpha+2}{N-2}\Big).
\tag{$\mathcal P_{\alpha,N}$}
\end{align}
\begin{remark}
\rm
Note that that if  $\max\{2,N/2\}<\alpha \leqslant N$, the two intervals defining  \ref{PalfN} intersect 
each other and \ref{PalfN}
becomes one interval $\Big(1,\frac{N-\frac4\alpha+2}{N-2}\Big)$.
\end{remark}

 \begin{enumerate}
      \item[\namedlabel{f1}{($h_1$)}]
      The nonlinearity $f$ is a function $\mathds{R}^{N} \times \mathds{R} \to \mathds{R}$ 
      such that for every $u\in \mathbb{R}$, the map 
      \begin{align*}
      \mathbb{R}^N\to \mathbb{R}, \quad x\mapsto f(x,u) \quad \text{ is Lebesgue measurable.} 
\end{align*}      
        \item[\namedlabel{f2}{($h_2$)}]
        There exist an $L^\infty_{\mathrm{loc}}(\mathbb{R}^N)$-function $\phi: \mathbb{R}^N\to [0,+\infty)$ with the property that $E_\phi = E_{b}$ and the norms
 $\|\cdot\|_{\phi}$ and $\|\cdot\|_{b_+}$ generate the same topology on $E_b$,
        $\alpha>\max\{2,N/2\}$, $p\in\,$\ref{PalfN}, 
        and $C_0>0$ such that
        for all $ x \in \mathds{R}^N$, $u,v\in \mathds{R}$, 
    \begin{align*}
                |f(x, u)-f(x,v)|  \leqslant C_0\big(1+ \phi(x)^{1/\alpha}\big)\left(1+ \left(|u|^{p-1} +|v|^{p-1}\right)\right)|u-v|.
\end{align*}
 \item[\namedlabel{f3}{($h_3$)}] For all $x\in\mathbb{R}^N$, $f(x,0)= 0$. 
        \item[\namedlabel{f4}{($h_4$)}]
        Let $F(x,u): = \int_0^u f(x, t) d t$. There exists a constant $\mu_0>2$ such that
        for $u\ne 0$,
                  \begin{align*}
                  0< \mu_0 F(x, u)  \leqslant  u f(x, u).
\end{align*}
            \item[\namedlabel{f5}{($h_5$)}]
            \label {h4}
               For almost every fixed $x\in \mathbb{R}^N$, the function
            $u\mapsto f(x,u)$ is nondecreasing. 
 \end{enumerate}

\begin{remark}
\label{rm15}
\rm
By  \ref{f2}--\ref{f4}, 
\begin{align*}
|f(x,u)| \leqslant C_0(1+ \phi^\frac1\alpha) (|u|+|u|^p) \quad \text{and} \quad
|F(x, u)| \leqslant \frac{C_0}{\mu_0}(1+ \phi^\frac1\alpha) (|u|^2+|u|^{p+1}).
\end{align*}
\end{remark}
\begin{remark}
\rm
Hypothesis \ref{f4} is known as a condition of the Ambrosetti-Rabinowitz type.
\end{remark}

\subsection{Main results}
Theorems \ref{teo1} and  \ref{teo} are the main results
of this work.
        \begin{theorem}\label {teo1}
 Let $b(x)\in \Upsilon_s(\mathds{R}^N)$ for some $s\in (2,2^*)$. Assume that 
 $a(x)$ satisfies {\rm \ref{f0a}} and $c(x) = \beta a(x)$ with $\beta>0$.
 Further assume that the nonlinearity $f$ satisfies conditions
   {\rm\ref{f1}-\ref{f4}}.
              Then, system \eqref{P} has a mountain pass type  solution.
         \end{theorem}

    \begin{theorem}\label {teo}
    Let $b(x)\in \Upsilon_2(\mathds{R}^N)$.  
        Assume that the coefficients $a(x)$ and $b(x)$ satisfy {\rm \ref{f0a}} and {\rm \ref{f0ab}},
 where $a>0$ a.e. and $c(x) = \beta a(x)$ with $\beta>0$. Further assume that the nonlinearity $f$ satisfies conditions
   {\rm \ref{f1}-\ref{f5}}.
Then system \eqref{P}   has three nontrivial weak solutions $(u_1, v_1), (u_2, v_2)$ and $(u_3, v_3)$, such that   
     \begin{align*}
     u_1> 0, \; v_1>0;  \quad u_2< 0, \; v_2<0; \quad  
     u_{3} \; \text{is sign-changing}.
\end{align*}
       \end{theorem}
\begin{corollary}
\label{cor-teo1-and-teo}
    Assume condition \ref{f22} below {\rm :}
    \begin{enumerate}
        \item[\namedlabel{f22}{$(h_2')$}] 
        There exist $p \in (1, 2^* - 1)$ and $C_0 > 0$ such that for all $x \in \mathds{R}^N$, $u, v \in \mathds{R}${\rm :}
        \begin{align*}
            |f(x, u) - f(x, v)| \leqslant C_0(1 + |u|^{p-1} + |v|^{p-1})|u - v|.
\end{align*}
    \end{enumerate}
    Then, Theorems \ref{teo1} and \ref{teo} remain valid
    if we replace {\rm \ref{f2}} with {\rm \ref{f22}}.
\end{corollary}

\subsection*{1.8. Our methodology}
To establish the existence and multiplicity of solutions for system \eqref{P}, we reduce it to a single nonlocal equation. Although this step is similar to 
\cite{MR3552796,MR3197242, MR0819214, MR4358261}, the aforementioned papers
deal with the constant coefficients case. 
To be able to handle variable and sign-changing coefficients, we introduce condition
\eqref{primeiroAutovalor} which is new in the context of FitzHugh-Nagumo systems.
More specifically, the existence of the linear continuous operator 
$u\mapsto (-\Delta + b(x))^{-1}(a(x) u)$ with variable 
and sign-changing coefficients $a$ and $b$ is  possible due to condition
\eqref{primeiroAutovalor}. 
Thus, we are able to solve the second equation of system \eqref{P}
with respect to $v$. Substituting $v=\beta (-\Delta + b(x))^{-1}(a(x) u)$
into the first equation of \eqref{P}
gives 
\begin{align}
\label {nonlocal-1111}
-\Delta u + \beta a(x)(-\Delta + b(x))^{-1}(a(x) u) = f(x,u).
\end{align}
Developing a variational framework for equation \eqref{nonlocal-1111} and proving compactness
of the embedding $E_b \hookrightarrow L_s(\mathds{R}^N)$, $s\in [2,2^*)$, can be regarded as important
 technical contributions of the paper. 
 These ingredients are 
 indispensable for proving the Palais–Smale condition for the associated energy functional
 and applying Weth's theorem \cite{MR2262254}.
 Both tasks are particularly challenging due to 
\begin{itemize}[leftmargin=*]
  \item[--] \textit{Sign-changing coefficients:} The functions \(a(x)\) and \(b(x)\) are merely assumed to be in \(L^\infty_{\mathrm{loc}}(\mathds{R}^N)\), and may change sign or vanish on large sets. 
  Unlike most works in the literature, which deal with positive constant coefficients,
  our framework allows for variable and sign-indefinite coefficients. These properties obstruct standard arguments.
  \item[--] \textit{Nonlocal operator with variable coefficients:} The presence of \((-\Delta + b(x))^{-1}a\), with \(a(x)\) and $b(x)$ changing sign, introduces significant complications in defining appropriate functional spaces and the energy functional associated with the problem. 
  More specifically, we introduce
  a Hilbert space \((E_b, \|\cdot\|_{ab})\), with the scalar product
   \begin{equation}
   \label {bform}
        \langle u,v \rangle_{ab}: =\int_{\mathds{R}^{N}} \nabla u \nabla v 
        + \beta a(x) u (-\Delta + b(x))^{-1}(av) \, \mathrm{d}x,
    \end{equation}
  associated with this operator. A primary difficulty in this analysis 
  lies in establishing conditions on  $a(x)$ and $b(x)$
 that permit both functions to assume positive and negative values while simultaneously ensuring that the bilinear form \eqref{bform} remains symmetric and positive definite.
  \item[--] \textit{Non-coercive coefficients}. We are able to recover  compactness
 of the embedding $E_b \hookrightarrow L_s(\mathds{R}^N)$, $s\in [2,2^*)$, due to condition \eqref{coecivo} on the coefficient $b(x)$, permitting 
 its non-coercivity. 
 \end{itemize}
 Furthermore, to show the 
 existence of componentwise positive and negative solutions, 
 we invoke a deep result of Weth \cite{MR2262254} which guarantees the existence of three nontrivial critical points
 for the energy functional: one positive, one negative, and one sign-changing, provided that certain geometric and compactness conditions are met. 
 Verifying the assumptions of Weth's theorem in our setting is highly nontrivial
 since it requires to check that 
 \begin{align}
 \label{polarK-cond}
 K^o\subset -K, 
\end{align}
 where $K = \{u\in E_b: u\geqslant 0 \; 
 \text{a.e.}\}$ is regarded as a closed, convex, pointed cone and $K^o$ is
 its polar cone in an appropriate Hilbert space. More specifically, 
 showing \eqref{polarK-cond} for a polar cone defined with respect to the scalar product \eqref{bform}  requires establishing a maximum principle 
 for the non-local operator
 on the left-hand side of \eqref{nonlocal-1111} which does not appear possible.
 Indeed, assuming that the right-hand side 
 of \eqref{nonlocal-1111} is non-negative, the standard technique requires a scalar multiplication by $u_-$, which leads to 
 \begin{align*}
\int_{\mathbb{R}^N} |\nabla u_-|^2 + \beta a(x) u_- (-\Delta + b(x))^{-1}(au) \, \mathrm{d}x
\leqslant 0.
\end{align*}
 The non-local nature of $(-\Delta + b(x))^{-1}(au)$ does not allow to identify
 the sign of the second term and hence, to conclude that $u_- = 0$ a.e.
 To overcome this difficulty, we 
 add a suitable linear term $e(x) u(x)$ to the both sides of equation \eqref{nonlocal-1111}
 and redefine the scalar product \eqref{bform}. Defining the polar cone
 with respect to the new scalar product allows us to establish \eqref{polarK-cond}.

Our methodology represents a substantial advancement compared to previous works 
on FitzHugh-Nagumo systems. For instance, in the papers \cite{MR3552796,MR3197242, MR0819214, MR4358261},
the authors analyze a system with constant coefficients, allowing them to exploit standard Sobolev embeddings. In contrast, our setting admits merely
$L^\infty_{\mathrm{loc}}(\mathbb{R}^N)$, possibly sign-changing, coefficients with no coercivity or decay conditions, thereby demanding a new analytical machinery. Furthermore, the number
$\lambda_1(\sigma)$ is defined for an $L^\infty_{\mathrm{loc}}(\mathbb{R}^N)$-function $\sigma$ 
with the infimum taken over ${\rm C}^\infty_{\mathrm{loc}}$–functions, whereas in \cite{MR1782990}, 
it is introduced only for continuous functions $\sigma$
and defined via an infimum over $E_\sigma$–functions. 
Finally, to the best of our knowledge, the use of Weth's theorem, together with establishing the maximum 
principle for a non-local operator and polar cone techniques, is new in the context of FitzHugh-Nagumo-type systems.

We also construct explicit examples of coefficients satisfying our assumptions but violating those in standard literature, 
further illustrating the generality of our results.

\subsection{Outline}
Section \ref{intro-1111} discusses hypotheses and main results of the paper, its methodology, and comparison with related papers.
Section \ref{auxiliary results} serves to introduce some auxiliary results
important for establishing the variational framework.
 Namely, in Subsection \ref{Properties of space}, we introduce and study the space $E_{\sigma}$ whose properties are used throughout the paper.
Section \ref{Variational framework} introduces the variational framework for our problem. 
In particular, in this section we study 
the nonlocal operator $u\mapsto (-\Delta + b(x))^{-1}(a(x) u)$, 
the Hilbert space with the scalar product \eqref{bform}, and introduce
the energy functional.
In Section \ref{theorem17}, we prove Theorem \ref{teo1} on the existence
of solution to system \eqref{P} with sign-changing coefficients. Finally,
in Section \ref{Proof of the main result}, we prove Theorem \ref{teo} on 
componentwise positive and negative solutions to system \eqref{P}.

    \section{Auxiliary results}\label {auxiliary results}
\subsection{Notation}\label {Notation}  From now on, we will use the following standard notation:
\begin{itemize}
    \item $\mathds{1}_{\text{}_{E}}:=$ the characteristic function of a set $E$.
    \item ${\rm C}_{0}^{\infty}(\mathds{R}^{N}):=$ the set of all infinitely differentiable functions with compact support in $\mathds{R}^{N}$.
    \item $D^{1,2}(\mathds{R}^{N}):=$
    the space of functions $u\in L_{2^{*}}(\mathds{R}^{N})$ which admit the first-order weak derivatives $\partial_i u\in L_2(\mathds{R}^{N})$, $i=1,2,\ldots,N$.
    \item The norms in $L_p\left(\mathds{R}^N\right)$, $D^{1,2}(\mathds{R}^{N})$ and $H^1(\mathds{R}^{N})$ will be denoted respectively by $\|\cdot\|_{{\scriptstyle L_p(\mathds{R}^N)}}$, $\|\cdot\|_{D^{1,2}}(\mathds{R}^N)$, and $\|\cdot\|_{H^{1}(\mathds{R}^N)}$,
    or simply by $\|\cdot\|_{{\scriptscriptstyle L_p}}$, $\|\cdot\|_{D^{1,2}}$, and $\|\cdot\|_{\scriptscriptstyle H^{1}}$,
    when it does not lead to misunderstanding.
    \item $u_{+}(x)=\max\{u(x),0\}$ and $u_{-}(x) = \min\{-u(x),0\}$.
    \item $ B_{r}(x)$ stands for the open ball of radius $r$ centered at $x \in \mathds{R}^{N}$.
     \item $ \displaystyle  E_{b}:=\left\lbrace  u\in D^{1,2}(\mathds{R}^{N}):  \int_{\mathds{R}^{N}} b_{+}(x) |u|^{2} \, \mathrm{d}x<+\infty \right\rbrace$. 
    \item $ \displaystyle  \|u\|_{b}^2:= \int_{\mathds{R}^{N}} |\nabla u|^2\, \mathrm{d}x +\int_{\mathds{R}^{N}}b(x)|u|^2\,  \mathrm{d}x$.
    \item $ \displaystyle  \|u\|_{ab}^2: = \int_{\mathds{R}^{N}} |\nabla u|^2\, \mathrm{d}x+ \int_{\mathds{R}^{N}} a(x) u S_{b}(u)\,  \mathrm{d}x$.
    \item $ \displaystyle \langle u, w\rangle_{ab}: =\int_{\mathds{R}^{N}} \nabla u \nabla w \, \mathrm{d}x+ \int_{\mathds{R}^{N}}a(x)uS_{b}(w) \, \mathrm{d}x$ \,  with \, $ S_b: = \beta(-\Delta + b(x))^{-1} a(x)$.
  \item $ \displaystyle \lambda_{1}(\sigma):= \inf \left\lbrace \frac{\int_{\mathds{R}^{N}}\left(|\nabla u|^2+\sigma(x)|u|^2\right) \, \mathrm{d}x}{\int_{\mathds{R}^{N}}|u|^2 \, \mathrm{d}x }: u\in {\rm C}_{0}^{\infty} \left(\mathds{R}^N \right)\setminus  \{0\} \right\rbrace$.
    \item $ \displaystyle  \nu_{s}(\sigma,\Omega):=\inf \left\{\frac{\int_{\Omega}\left(|\nabla u|^2+\sigma(x)|u|^2\right) \, \mathrm{d}x}{\left(\int_{\Omega}|u|^s \, \mathrm{d}x \right)^{2/s}}: u \in H^{1}_{0}(\Omega) \setminus  \{0\} \right\}$, \qquad $s\in [2,2^*)$,\\
    where $\Omega\subset \mathds{R}^N$ is an open set.
\end{itemize}

\subsection{Properties of the space $E_\sigma$}\label {Properties of space}

    This subsection is dedicated to studying the space
    \begin{align*}
    E_{\sigma}:=\left\lbrace  u\in D^{1,2}(\mathds{R}^{N}):  \int_{\mathds{R}^{N}} \sigma_{+}(x) |u|^{2} \,
    \mathrm{d}x<+\infty \right\rbrace.
\end{align*}
    The next two lemmas are fundamental for this work.
    Due to these lemmas,  the space
    $E_{\sigma}$ turns out to be a Banach space with the norm formally given by
   \begin{align*}
     \|u\|_{\sigma}^2= \int_{\mathds{R}^{N}} |\nabla u|^2\, \mathrm{d}x +\int_{\mathds{R}^{N}}\sigma(x)|u|^2\,  \mathrm{d}x.
\end{align*}
Since $\sigma$ can take negative values, a priori we do not
know if $\|\cdot\|_{\sigma}$ is a norm. The latter fact will be proved in Lemma
\ref{sg-norm-1111} under the assumption $\lambda_1(\sigma)>0$, 
while the preceding lemmas  are also important for this proof.
Therefore, for now, instead of
$\|\cdot\|^2_{\sigma}$, we will use the auxiliary function
    \begin{equation*}
    \mathscr{N}_{\sigma}: E_{\sigma}\rightarrow \mathds{R}, \quad
     \mathscr{N}_{\sigma}(u)= \int_{\mathds{R}^{N}}\big( |\nabla u|^2 +\sigma(x)u^2\big)  \mathrm{d}x.
    \end{equation*}
    \begin{lemma}
    \label{seminorm-1111}
    The function
    \begin{align*}
    \|\cdot\|_{\sigma_+}:=\sqrt{\Ngp}
\end{align*}
    is a seminorm in $E_\sigma$. If $\sigma_+ \ne 0$ a.e., then $ \|\cdot\|_{\sigma_+}$ is a norm.
    \end{lemma}
    \begin{proof}
    Verification of the seminorm properties for $\|\cdot\|_{\sigma_+}$ is straightforward. 
    
    Further, if $\|u\|_{\sigma_+}= 0$, then
    $\|\nabla u\|_{L_2} = 0$, and hence $u=const$.
 Further, $\|u\sqrt{\sigma_+}\|_{L_2} = 0$ implies $const = 0$.
    Hence, $\|u\|_{\sigma_+} = 0$ if and only if $u=0$ a.e.
    \end{proof}

    \begin{lemma}\label {E_subset_H}
     Let $\sigma\in L^\infty_{\rm loc}(\mathbb{R}^N)$ and $\lambda_1(\sigma_{+}) > 0$. Then,
        $E_{\sigma}$  is a subset of $H^{1}(\mathds{R}^{N})$.
    \end{lemma}
    \begin{proof}
    In what follows, we write $B_r$ for $B_r(0)$ to simplify notations.
    Suppose that our claim is false. Then, there exists $u\in E_{\sigma}$ such that $u \not\in L_2(\mathds{R}^N)$.  Consider the function $\phi^r$ given by
        \begin{equation}
        \label{phir1111}
            \phi^{r}(x)=\left\lbrace
            \begin{array}{cl}
                1,& \text{if }|x|\leqslant r,\\
                2-\dfrac{|x|}{r}, & \text{if } r< |x|<2r,\\
                0,& \text{if }|x|\geqslant 2r
            \end{array}\right.
        \end{equation}
          and define
          \begin{equation}
          \label{ur11111}
              u^r(x) := u(x)\cdot \phi^{r}(x).
          \end{equation}
    Note that $u^r$ has compact support, $|u^r| \rightarrow  |u|$
    and $|u^r| \leqslant  |u|$
    on $\mathds{R}^{N}$. It is clear that  $u^r\in H^{1}_{0}(B_{r})$ and $\|u^r\|_{\scriptscriptstyle L_{2}}\rightarrow \infty$ as $r\rightarrow \infty$. Furthermore,
         \begin{equation}
         \label {nabu-1111}
	   |\nabla u^{r}(x)|\leqslant |\nabla u(x)|+ |u(x)|\cdot \mathds{1}_{B_{2r}\setminus B_{r}}(x)\cdot\dfrac{1}{r}.
        \end{equation}
       By \eqref{nabu-1111} and H\"older's inequality, there exists a constant $C>0$ such that
\begin{equation}\label {eq2.E_subset_H}
            \begin{split}
                \int_{\mathds{R}^{N}} |\nabla u^{r}|^{2}\, \mathrm{d}x
                &\leqslant 2\int_{\mathds{R}^{N}} |\nabla u|^{2}\, \mathrm{d}x+
                \frac{2} {r^{2}}\int_{\mathds{R}^N} \mathds{1}_{B_{2r}\setminus  B_r} |u|^{2}\, \mathrm{d}x\\
                &\leqslant 2\int_{\mathds{R}^{N}} |\nabla u|^{2}\, \mathrm{d}x+\frac{2}{r^{2}}\left( \int_{\mathds{R}^{N}} |u|^{2^{*}}\, 
                \mathrm{d}x\right)^{\frac{2}{2^{*}}}
                \big|{B_{2r}\setminus  B_r}\big|^{\frac{2}{N}} \\
                &= 2\int_{\mathds{R}^{N}} |\nabla u|^{2}\, \mathrm{d}x+ \frac{C}{r^{2}}\left( \int_{\mathds{R}^{N}} |u|^{2^{*}}\, \mathrm{d}x\right)^{\frac{2}{2^{*}}} \left(  r^{N} \right) ^{\frac{2}{N}}= 2\|\nabla u\|^{2}_{\scriptscriptstyle L_2}+ C\|u\|^{2}_{\scriptscriptstyle L_{2^{*}}}.
	   \end{split}
        \end{equation}
        Above, we have applied the H\"older inequality with conjugate exponents $2^*/2 $ and $2^*/(2^*-2)=N/2$.
       By \eqref{eq2.E_subset_H},
 \begin{equation}\label {asdf}
        \begin{split}
             \frac{\|u^r\|^2_{\sigma_{+}}}{\| u^r\|^{2}_{\scriptscriptstyle L_2}}&=\frac{1}{\| u^r\|^{2}_{\scriptscriptstyle L_2}}
             \int_{\mathds{R}^{N}}\left(|\nabla u^{r}|^2+\sigma_{+}(x)|u^{r}|^2\right) \, \mathrm{d}x\\
             &\leqslant \frac{1}{\| u^r\|^{2}_{\scriptscriptstyle L_2}}\left(2\|\nabla u\|^{2}_{\scriptscriptstyle L_2}+ 
             C\|u\|^{2}_{\scriptscriptstyle L_{2^{*}}}+\int_{\mathds{R}^{N}}\sigma_{+}(x)|u|^2 \, \mathrm{d}x\right)  \rightarrow 0 \quad \text{ as } r 
             \rightarrow \infty.
        \end{split}
        \end{equation}
         Further, by the density of ${\rm C}_{0}^{\infty}(B_{2r})$   in $H^{1}_{0}(B_{2r})$, there exists a sequence $\varphi_{r,n} \in {\rm C}_{0}^{\infty}(B_{2r})$ such that 
         $\lim_{n\to\infty} \|\varphi_{r,n}-u^{r}\|_{H^{1}} = 0$. Therefore,
        \begin{equation*}
      \begin{split}
        \| \varphi_{r,n}-u^{r}\|^2_{\sigma_{+}}\leqslant  \|\nabla \varphi_{r,n}-\nabla u^{r}\|_{\scriptscriptstyle L_2}^{2}
        +\|\sigma\|_{\scriptscriptstyle L^{\infty}(\,\overline{B_{2r}}\,)}
        \|\varphi_{r,n}-u^{r}\|_{\scriptscriptstyle L_2}^{2} \to 0
        \quad \text{as} \; \; n\to \infty.
      \end{split}
      \end{equation*}
      By the definition of $\lambda_1$,
      \begin{align*}
           \sqrt{\lambda_{1}(\sigma_{+})}\leqslant
           \frac{\|\varphi_{r,n}\|_{\sigma_+}}{\|\varphi_{r,n}\|_{\scriptscriptstyle L_2}}
           \leqslant \frac{\|u^r\|_{\sigma_+}\bigg(1+ \frac{\|u^r-\varphi_{r,n}\|_{\sigma_+} }{\|u^r\|_{\sigma_+} }\bigg)}
           {\|u^r\|_{\scriptscriptstyle L_2}\bigg(1- \frac{\|u^r-\varphi_{r,n}\|_{\scriptscriptstyle L_2}}{\|u^r\|_{\scriptscriptstyle L_2}}\bigg)}.
\end{align*}
      Passing to the limit as $n\to \infty$ and then to the limit as $r\to\infty$, by \eqref{asdf}, we obtain that $\lambda_1(\sigma_+) \leqslant 0$.
      This is a contradiction. Hence, $u \in L_2(\mathds{R}^N)$.
 \end{proof}
\begin{lemma}
Let $\sigma\in L^\infty_{\rm loc}(\mathbb{R}^N)$ and $\lambda_1(\sigma)>0$. Then, $\sigma_+ \ne 0$ a.e., and hence,
$ \|\cdot\|_{\sigma_+}$ is a norm.
\end{lemma}
\begin{proof}
Suppose $\sigma_+ = 0$ a.e. Consider the function $\sigma_{gauss}$ defined by \eqref{gauss} in Example \ref{example3333}.
As it is shown in Example \ref{example3333}, $\lambda_1(\sigma_{gauss})=0$.
In the proof, one only uses estimate \eqref{eq2.E_subset_H} obtained in Lemma \ref{E_subset_H}.
From  condition \eqref{primeiroAutovalor} it follows that
\begin{align*}
\lambda_1(\sigma) \leqslant \lambda_1(\sigma_{gauss}) = 0
\end{align*}
which contradicts to our assumptions. Hence, $\sigma_+ \ne 0$ a.e. and $ \|\cdot\|_{\sigma_+}$ is a norm
by Lemma \ref{seminorm-1111}.
\end{proof}
  \begin{lemma}\label {dense}
         Let $\sigma\in L^\infty_{\rm loc}(\mathbb{R}^N)$ and $\lambda_1(\sigma_{+}) > 0$. Then,  ${\rm C}_{0}^{\infty}(\mathds{R}^N)$ 
         is dense in $(E_{\sigma}, \|\cdot\|_{\sigma_+})$.
    \end{lemma}
    \begin{proof}
       Let $u \in E_{\sigma}$.
       First, we consider the case when $u$ has compact support in $\mathds{R}^{N}$.
        This means that there exists an open ball $B\subset\mathbb{R}^N$ such that $\operatorname{supp}(u) \subset B$.
        It is clear that  $u\in H^{1}_{0}(B)$. By the density of ${\rm C}_{0}^{\infty}(B)$
        in $H^{1}_{0}(B)$, one  finds
         a sequence $\varphi_n \in {\rm C}_{0}^{\infty}(B)$ such that
        $\lim_{n\to\infty} \|\varphi_n-u\|_{H^{1}(B)} = 0$. We have
        \begin{equation*}
      \begin{split}
         \| \varphi_{n}-u\|_{\sigma_+}\leqslant  \|\nabla \varphi_{n}-\nabla u\|_{\scriptscriptstyle L_2}^{2} +\|\sigma\|_{\scriptscriptstyle L^{\infty}(\overline B)}\|\varphi_{n}-u\|_{\scriptscriptstyle L_2}^{2} \to 0 \quad \text{as} \;\; n\to \infty.
      \end{split}
       \end{equation*}
 Consider the case when the support of $u$ is not necessarily compact. Define $u^r$ by \eqref{ur11111} with the help of the function $\phi^r$ given by \eqref{phir1111}.
        Since   $u\in E_{\sigma}$,
        we have $\sigma_+ |u^r|^2\leqslant \sigma_+ |u|^2 \in L_1(\mathds{R}^{N})$. Therefore,
        \begin{equation}
        \label{convr-1111}
         \int_{\mathds{R}^{N}} \sigma_+(x) |u^r(x)-u(x)|^2 \, \mathrm{d}x \rightarrow  0 \quad \text{as}\;\;
         r\to \infty.
     \end{equation}

     By Lemma \ref{E_subset_H}, $u\in H^1(\mathbb{R}^N)$. Hence,
     \begin{equation*}
        \begin{split}
            \int_{\mathds{R}^{N}}|\nabla u^r(x)-\nabla u(x)|^{2}\, \mathrm{d}x &=\int_{\mathds{R}^{N}}|u\nabla \phi^{r} +\phi^{r} \nabla u -\nabla u|^{2}\, \mathrm{d}x\\
            &\leqslant \dfrac{2}{r^2}\int_{\mathds{R}^{N}} |u |^2 \,  \mathrm{d}x+2\int_{\mathds{R}^{N}} |(\phi^{r}-1) \nabla u |^2 \,  \mathrm{d}x \to 0, \quad r\to \infty.
        \end{split}
     \end{equation*}
     Thus, $\|u^r - u\|_{\sigma_+} \to 0$ as $r \to \infty$. This concludes the proof.
     \end{proof}
     \begin{remark}
     \label{rem1111}
     \rm
     Note that one can find a sequence $\varphi_n\in {{\rm C}}^\infty_0(\mathds{R}^N)$
        which approximates $u\in E_\sigma$ simultaneously
        in $E_\sigma$ and $H^1(\mathbb{R}^N)$.
        Indeed, since $u\in L_2(\mathbb{R}^N)$ 
        ($E_\sigma$ is a subspace of $H^1(\mathbb{R}^N)$), convergence
        \eqref{convr-1111} holds also with $\sigma_+ \equiv 1$.   
     \end{remark}

    \begin{lemma}\label {eqlambda}
        Let $\sigma\in L^\infty_{\rm loc}(\mathbb{R}^N)$ and $\lambda_1(\sigma) > 0$.
        Define
        \begin{equation*}
            \Tilde{\lambda}_{1}(\sigma):= \inf \left\lbrace
            \int_{\mathds{R}^{N}}\left(|\nabla u|^2+\sigma(x)|u|^2\right) \, \mathrm{d}x:
            \|u\|_{\scriptscriptstyle L_2}=1 \text{ and } u\in E_{\sigma}\right\rbrace.
        \end{equation*}
  Then,  $\tilde\lambda_1(\sigma)=\lambda_{1}(\sigma)$.
  Moreover,  $\lambda_1(\sigma)>0$ if and only if
  $\tilde\lambda_1(\sigma)>0$.
    \end{lemma}
    \begin{proof}
    Since ${\rm C}_{0}^{\infty}(\mathds{R}^{N}) \subset E_{\sigma}$, we have $\lambda_1(\sigma) \geqslant \tilde{\lambda}_1(\sigma)$. 
    
    To show that $\lambda_1(\sigma) \leqslant \tilde{\lambda}_1(\sigma)$, let $u \in E_{\sigma}$  be such that $\|u\|_{\scriptscriptstyle L_2} = 1$.  By Lemma \ref{dense}, there exists a sequence $\{\varphi_n\} \subset {\rm C}_{0}^{\infty}(\mathds{R}^{N})$ such that
    $\|\varphi_n-u\|_{\sigma_+} \to 0$ and $\|\varphi_n-u\|_{H^1(\mathbb{R}^N)} \to 0$.  Define the sequence $\psi_n = \varphi_n / \|\varphi_n\|_{L_2(\mathds{R}^N)}$, 
    so that $\|\psi_n\|_{L_2(\mathds{R}^N)} = 1$. Consider the following seminorm
      \begin{align*}
        \|u\|_{\sigma_-}^{^\sim}
        = \left(\int_{\mathbb{R}^N} |\sigma_-||u|^2 dx \right)^\frac12.
\end{align*}
        By the triangle inequality for seminorms,
        \begin{align*}
            \big| \|\psi_n\|_{\sigma_-}^{^\sim} - \|u\|_{\sigma_-}^{^\sim}\big|
            \leqslant \|\psi_n - u\|_{\sigma_-}^{^\sim} \leqslant 
            \esssup_{\mathbb{R}^N}|\sigma_-|\,\|\psi_n - u\|_{\scriptscriptstyle L_2}.
\end{align*}
   Hence,
        \begin{align*}
            \lim_{n\to\infty} \int_{\mathbb{R}^N}|\sigma_-||\psi_n|^2 dx
            = \int_{\mathbb{R}^N}|\sigma_-||u|^2 dx,
\end{align*}
        and therefore,
        \begin{align*}
            \Ng(\psi_n) = \|\psi_n\|^2_{\sigma_+} -   \int_{\mathbb{R}^N}|\sigma_-||\psi_n|^2 dx \to \Ng(u).
\end{align*}
        Note that if $\sigma_- = 0$ a.e., the above argument
        on convergence $\Ng(\psi_n) \to \Ng(u)$ still holds.
     Thus,  we have $\Ng (u)=\lim_{n\to\infty} \Ng( \psi_n) \geqslant \lambda_1(\sigma)$.   Hence, $\tilde{\lambda}_1(\sigma)  \geqslant  \lambda_1(\sigma)$.
      \end{proof}

    \begin{lemma}
    \label{sg-norm-1111}
         Let $\sigma\in L^\infty_{\rm loc}(\mathbb{R}^N)$ and $\lambda_1(\sigma) > 0$. Then, the bilinear form $\langle \cdot, \cdot \rangle_{\sigma}: E_{\sigma} \times E_{\sigma} \to \mathds{R}$, defined by
        \begin{equation}
        \label{sc-prod-1111}
        \langle u, w \rangle_{\sigma} = \int_{\mathds{R}^N} \nabla u \cdot \nabla w \, \mathrm{d}x + \int_{\mathds{R}^N} \sigma(x) u w \, \mathrm{d}x
        \end{equation}
        defines an inner product in $E_\sigma$. In particular, $\|\cdot\|_{\sigma} := \sqrt{\mathscr{N}_{\sigma}(\cdot)}$ is a norm in $E_{\sigma}$.
    \end{lemma}

    \begin{proof}

    Let $u, v, w \in E_{\sigma}$ and $t \in \mathds{R}$. We have $\langle u, v \rangle_{\sigma} = \langle v, u \rangle_{\sigma}$, and
    \begin{equation*}
        \langle tu + v, w \rangle_{\sigma} = t \langle u, w \rangle_{\sigma} + \langle v, w \rangle_{\sigma}.
    \end{equation*}
 Further, by Lemma \ref{E_subset_H},  $E_\sigma \subset L_2(\mathds{R}^N)$. Consequently,
 by Lemma \ref{eqlambda},
    \begin{equation*}
        \mathscr{N}_{\sigma}(u) \geqslant \tilde{\lambda}_1(\sigma) \|u\|^2_{\scriptscriptstyle L_2}\geqslant 0.
    \end{equation*}
     Moreover, if $\mathscr{N}_{\sigma}(u)=0$, then $u=0$ a.e. The opposite statement is obvious. 
Thus, we have proved that
$\langle u, w \rangle_{\sigma}$ is a scalar product.
Therefore, $\|\cdot\|_{\sigma} := \sqrt{\mathscr{N}_{\sigma}(\cdot)}$ is a norm in $E_{\sigma}$.
    \end{proof}
An embedding similar to \eqref{imbedding} in Lemma \ref{Embedded}  was obtained in \cite[Lemma 2.1]{MR1782990}. We include here the proof of the lemma
since it is obtained for a broader class of functions $\sigma$ and, in particular, under the weaker condition \eqref{primeiroAutovalor}.
    \begin{lemma} \label {Embedded}
    Let $\sigma\in \mathcal T(\mathbb{R}^N)$.
       Then, the canonical embedding
       \begin{align}
       \label{imbedding}
       \iota:\, E_{\sigma} \hookrightarrow H^1(\mathbb{R}^N)
\end{align}
         is continuous.
    \end{lemma}
    \begin{proof}
    Recall that by Lemma \ref{E_subset_H}, $E_\sigma \subset H^1(\mathbb{R}^N)$. Thus, it is sufficient to prove that there exists a constant $C>0$ such that
	\begin{equation} \label {NablaEmbedded}
\|u\|_{H^1(\mathbb{R}^N)} \leqslant C \|u\|_\sigma \qquad \forall \; u\in E_{\sigma}.
	\end{equation}
   Suppose this is not true. Then, one can find a sequence
        $v_n\in E_\sigma$, $v_n\ne 0$, such that
        $\|v_n\|_{H^1(\mathbb{R}^N)} \geqslant  n \|v_n\|_\sigma$.
  Note that for a sequence $u_n = v_n/\|v_n\|_{H^1(\mathbb{R}^N)}$, it holds that
       \begin{align*}
       \|u_n\|_{H^1(\mathbb{R}^N)} = 1 \quad \text{and} \quad
       \|u_n\|_\sigma \leqslant \frac1n.
\end{align*}
        Let us show that $\|u_n\|_{L_2(\mathbb{R}^N)}$ cannot converge to 0.
     Suppose $\lim\limits_{n\to \infty}\|u_n\|^2_{L_2(\mathbb{R}^N)} = 0$.
        Then,
        \begin{multline*}
          (\essinf_{\mathbb{R}^N} \sigma_-)
            \|u_n\|^2_{L_2(\mathbb{R}^N)}  \leqslant \int_{\mathds{R}^{N}} \sigma_-(x) u_n^2 \, \mathrm{d}x  \leqslant
            \int_{\mathds{R}^{N}} \sigma(x) u_n^2 \, \mathrm{d}x = \|u_n\|^2_\sigma 
            + \|u_n\|^2_{L_2(\mathbb{R}^N)}
            -\|u_n\|^2_{H^1(\mathbb{R}^N)} \\
            \leqslant \frac{1}{n^2} + \|u_n\|^2_{L_2(\mathbb{R}^N)} -1,
\end{multline*}
        which is a contradiction. Hence,
        there exists a subsequence
        $\left\{u_{n_k}\right\} \subset \{u_n\}$ and
        a constant $C>0$ such that $\|u_{n_k}\|^{2}_{\scriptscriptstyle L_2}> C$.
        This implies
        \begin{equation*}
           0<\lambda_{1}(\sigma)\leqslant \int_{\mathds{R}^{N}}\left( \frac{| \nabla u_{n_k}|^{2}}{\|u_{n_k}\|_{\scriptscriptstyle L_{2}}^{2}} 
           + \sigma(x) \frac{u_{n_k}^{2}}{\|u_{n_k}\|_{\scriptscriptstyle L_{2}}^{2}}  \right) \, \mathrm{d}x\leqslant 
           \frac{1}{n_{k}^2\|u_{n_{k}}\|^{2}_{\scriptscriptstyle L_2}}<\frac{1}{Cn^2_{k}} \to 0, \quad k\to \infty,
        \end{equation*}
    which is impossible. Hence, \eqref{NablaEmbedded} is true and the proof is complete.
 \end{proof}
    \begin{remark}
    \label{rm2.6-1111}
    \rm
    We remark that by a version of the
    Sobolev embedding theorem (see, e.g., \cite[Theorem 1.8]{zbMATH00921895}), the  continuous canonical embedding
    $H^1(\mathbb{R}^N) \hookrightarrow L_s(\mathbb{R}^N)$ holds for
    $s\in [2,2^*]$. 
    Therefore, if $\sigma$ satisfies assumptions of Lemma \ref{Embedded},
the aforementioned lemma implies that the canonical embedding
     \begin{align}
       \label{imbeddingLp}
       \iota:\, E_{\sigma} \hookrightarrow L_s(\mathbb{R}^N) 
\end{align}
        holds for all $s\in [2, 2^{*}]$ and is continuous.
    \end{remark}

    \begin{lemma}\label {Lemaeqnorms}
        Let $\sigma\in \mathcal T(\mathbb{R}^N)$.   Then, the norms $\|\cdot \|_{\sigma}$ and $\|\cdot \|_{\sigma_+}$ are equivalent.
    \end{lemma}

    \begin{proof}
        Let $u \in E_{\sigma}$. Clearly,
        $\|u\|_{\sigma_+}^{2} \geqslant \|u\|_{\sigma}^{2}$.
        On the other hand, by Lemma \ref{Embedded}, there exists a constant $C>0$ such that
        \begin{equation*}
            \|u\|_{\sigma}^{2} = \|u\|_{\sigma_+}^{2} - \int_{\mathds{R}^N} |\sigma_{-}(x)| u^2 \, \mathrm{d}x
            \geqslant \|u\|_{\sigma_+} - \esssup_{\mathbb{R}^n}|\sigma_-| \|u\|_{\scriptscriptstyle L_2}^{2}
            \geqslant \|u\|_{\sigma_+}^{2} - C \|u\|_{\sigma}^{2}.
        \end{equation*}
    \end{proof}

      \begin{lemma}
      \label{equiv-1111}
    Let $b\in \mathcal T(\mathbb{R}^N)$ and $a(x)$ satisfy {\rm \ref{f0a}}.
        Then, the norms $\|\cdot\|_{a}$ and
        $\|\cdot\|_{b}$ are equivalent, that is,
         there exist constants $C_1,C_2>0$ such that for all $u\in E_b$,
        $$
            C_1 \|u\|_b \leqslant \|u\|_a \leqslant C_2 \|u\|_b.
        $$
    \end{lemma}
    \begin{proof}
  By Lemma \ref{Lemaeqnorms},
    it suffices to prove the equivalence of the norms
    $\|\cdot\|_{a_+}$ and $\|\cdot\|_{b_+}$.
    Since these norms generate the same topology on $E_a = E_b$,
    the identity map
    \begin{align*}
    I: (E_a,\|\cdot\|_{a_+}) \to (E_a,\|\cdot\|_{b_+})
\end{align*}
    and its inverse $I^{-1}$ are continuous. Therefore,
    $I$ and $I^{-1}$ are bounded.
       \end{proof}

     \begin{lemma} \label {Completo}
         Let $\sigma\in \mathcal T(\mathbb{R}^N)$. Then, the space 
         $(E_{\sigma}, \|\cdot\|_\sigma)$ is separable and complete, and therefore,  it  is a Hilbert space.
    \end{lemma}
    \begin{proof}
    First of all, 
    by Lemma \ref{sg-norm-1111}, the bilinear form \eqref{sc-prod-1111} is a scalar product in $E_\sigma$.
    
     It remains to show that $E_\sigma$ is separable and complete.  Note that
    \[
E_\sigma = H^1(\mathds{R}^N) \cap L^2\big(\mathds{R}^N, \sigma_+(x)\, dx\big)
\]
This is an immediate consequence of Lemmas \ref{Embedded} and \ref{Lemaeqnorms}. 
Moreover, Lemmas \ref{Embedded} and \ref{Lemaeqnorms} imply that the norm $\|\cdot\|_\sigma$ is equivalent to 
the norm $\|\cdot\|_{H^1(\mathbb{R}^N)}+\|\cdot\|_{L_2(\mathbb{R}^N,\sigma_+ dx)}$.

 According to \cite[Section 3]{MR924157},
    the weighted Hilbert space $\big(L_2(\mathbb{R}^N), \sigma_+(x)dx\big)$ is  complete.
    Take a Cauchy sequence $\{u_n\} \subset E_\sigma$.
   Then, $\{u_n\}$ is also a Cauchy sequence
    in $H^1(\mathbb{R}^N)$ and $\big(L_2(\mathbb{R}^N), \sigma_+(x)dx\big)$, and therefore has a limit in these spaces.
    More specifically, there exists $u\in H^1(\mathbb{R}^n)$ such that 
    $\lim_{n\to\infty} \|u_n - u\|_{L_2(\mathbb{R}^N)} = \lim_{n\to\infty} \|\nabla u_n - \nabla u\|_{L_2(\mathbb{R}^N)} = 0$.
    By the local boundedness of $\sigma_+$, $\lim_{n\to\infty}\|u_n - u\|_{(L_2(B_R(0),\,\sigma_+dx)} = 0$ for any ball
    $B_R(0)\subset \mathbb{R}^N$. Consequently, $\lim_{n\to\infty}\|u-u_n\|_{\sigma_+} = 0$. Therefore,  $u\in E_\sigma$
    is the limit of $u_n$ in $E_\sigma$.    Thus, $(E_\sigma,\|\cdot\|_\sigma)$ is  complete. 
    
    Let us prove separability.
For each rational $r \in \mathds{Q}^+$, the space $H^1_0(B_{2r})$ is separable. Let $D_r$ be a countable dense subset of $H^1_0(B_{2r})$. 
By the local boundedness of $\sigma_+$, $D_r\subset E_\sigma$. Set $D := \cup_{r \in \mathds{Q}^+} D_r$.
Given $u \in E_\sigma$ and $\varepsilon > 0$, by the argument used in the proof of Lemma \ref{dense},
there exists $r \in \mathds{Q}^+$ such that 
$\|u - \phi_r u\|_{\sigma_+} < \frac{\varepsilon}{2}$,
where $\phi_r$ is defined by \eqref{phir1111}.
Since $\phi_r u \in H^1_0(B_{2r})$, there exists  $v \in D_{r}$  such that
\[
\|\phi_r u - v\|_{\sigma_+} \leqslant 
(\esssup_{B_{2r}}|\sigma_+|+1) \|\phi_r u - v\|_{H^1_0(B_{2r})} < \frac{\varepsilon}{2}.
\]
Therefore, $\|u - v\|_{\sigma_+}<\varepsilon$ and the proof is complete.
\end{proof}

   \subsection{Important examples}\label {Important examples}
   As we have already mentioned, FitzHugh-Nagumo-type systems considered in the literature 
   have constant coefficients. On the other hand,
   most related elliptic systems with spatially varying coefficients, see, 
   e.g., \cite{rabinowitz1992class}, assume positive and coercive coefficients. 
    In this subsection, we give an example (Example
    \ref{Example_Sigma}) of a function from $\Upsilon_{s}(\mathds{R}^N)$
    which takes negative values and is not coercive.
    In particular, this shows that the class of functions
    $\Upsilon_{s}(\mathds{R}^N)$ is non-empty.

In Example \ref{Example_Sigma}, we will make use of Lemma
\ref{Th1.4-sirakov}, which is the subject of Theorem 1.4 in \cite{MR1782990}.
\begin{lemma}
\label{Th1.4-sirakov}
Assume for any $M,R>0$ and any sequence $z_n\to \infty$ as $n\to \infty$,
\begin{align*}
\lim_{n\to \infty} |\Omega_M \cap B_R(z_n)| = 0, \quad 
\text{where} \quad \Omega_M: =\{ x\in\mathbb{R}^N: \sigma(x)<M\}.
\end{align*}
Then, condition \eqref{coecivo} is fulfilled for all $s\in [2,2^*)$.
\end{lemma}
   \begin{example}
   \label {Example_Sigma}
    \rm
    Here we give an example of a non-coercive function
    $\sigma \in \Upsilon_{s}\left(\mathds{R}^N\right)$, for all $s\in[2,2^*) $, that admits negative values.
    Fix numbers $r,\kappa >0$ and
    define a continuous function
      $\sigma: \mathds{R}^{N} \to \mathds{R}$ as follows. Let $x\in\mathds{R}$, $y\in\mathds{R}^{N-1}$ be such that
      $z=(x,y)$. Further let
      \begin{equation}
      \label{class-sigma}
      \sigma(z)= \begin{cases}
		     -\frac{\mu_{r}}{2r},  &\text { if } |z|\leqslant r/2, \\
                1 + \kappa^2(1+|x|)^{2}|y|^2, &\text { if } |z|\geqslant r, \\ 
                \sigma(z) \geqslant -\frac{\mu_{r}}{2r}
                &\text{ in the ring} \;\;
                B_{\frac{r}2,r}: = \{\frac{r}2 < |z| < r\}.
\end{cases}		 
	\end{equation}
  Above, $\mu_{r}>0$ is the constant appearing in
  the Poincar\'e inequality with respect to the domain $B_{2r}(0)$, which we denote for simplicity by $B_{2r}$, i.e.,
    \begin{equation*}
       \mu_{r}:=\inf_{ u \in H^{1}_{0}(B_{2r})\setminus \{0\}} \left\{\frac{\int_{B_{2r}}|\nabla u|^2 \, \mathrm{d}x }{\int_{B_{2r}}|u|^2 \, \mathrm{d}x} \right\}.
    \end{equation*}
    Note that $\sigma(z)$ is not coercive. Indeed, recalling that
     $z=(x,y)\in \mathbb{R} \times \mathbb{R}^{N-1}$,
    we observe that
    \begin{align*}
\lim_{|x|\to\infty}|\sigma(x,0)| = 1 \ne \infty.
\end{align*}
       In part (1) below, we will show that $\sigma$ satisfies \eqref{primeiroAutovalor}, and in part (2), we will  demonstrate that $\sigma$ satisfies \eqref{coecivo} for all $s\in [2,2^*)$.
    \begin{enumerate}[leftmargin=*]
        \item
            To prove that $\sigma$ satisfies \eqref{primeiroAutovalor},
            by Lemma \ref{eqlambda},
            it suffices to show that there exists a constant
            $C > 0$ such that $\tilde \lambda_1(\sigma)\geqslant C$. This is equivalent
            to 
            \begin{align}
            \label{ineq-1111}
            \|u\|_{\sigma} \geqslant C\|u\|_{\scriptscriptstyle L_2} \qquad \forall \;\; u \in E_{\sigma}.
            \end{align}
            Let $u \in E_{\sigma}$ and let $\phi^r$ be
            defined as in \eqref{phir1111} with $r>0$ fixed. By \eqref{eq2.E_subset_H},
             \begin{equation*}
                    \int_{\mathds{R}^{N}} |\nabla ( \phi^{r}u)|^2 \,  \mathrm{d}x 
                    \leqslant \frac{2}{r^{2}}\int_{B_{2r}\setminus B_{r}} |u |^2 \,  \mathrm{d}x+2\int_{\mathds{R}^N} |\nabla u |^2 \,  
                    \mathrm{d}x.
             \end{equation*}
            Therefore,
             \begin{equation*}
                 \begin{split}
                    2\|u\|_{\sigma}^{2}&
                     \geqslant 2 \|u\|_{\sigma}^{2}+\int_{\mathds{R}^{N}} |\nabla ( \phi^{r}u)|^2 \,  \mathrm{d}x- \frac{2}{r^{2}}\int_{B_{2r}\setminus B_{r}}  |u |^2 \,  \mathrm{d}x-2\int_{\mathds{R}^N} |\nabla u |^2 \,  \mathrm{d}x\\
                     &=2\int_{\mathds{R}^{N}} \sigma(x) |u|^2\,  \mathrm{d}x +\int_{\mathds{R}^{N}} |\nabla ( \phi^{r} u)|^2 \,  \mathrm{d}x - \frac{2}{r^{2}}\int_{B_{2r}\setminus B_{r}}  |u |^2 \,  \mathrm{d}x .
                 \end{split}
             \end{equation*}
             Since $\sigma(x) \geqslant -\mu_{r}/2r$ on $B_r$ and
             $\sigma(x) \geqslant 1$ on $\mathds{R}^{N}\setminus B_{r}$, we obtain
               \begin{equation}\begin{aligned}
               \label{est-2222}
                    2\|u\|_{\sigma}^{2}
                     \geqslant &\, -\frac{\mu_{r}}{r}\int_{B_{r }}  |u|^2\,  \mathrm{d}x
                     +\left(2-  \frac{2}{r^{2}}\right)\int_{B_{2r}\setminus B_{  r }}  |u|^2\,  \mathrm{d}x
                     + 2\int_{\mathds{R}^{N}\setminus B_{2r}}  |u|^2\,  \mathrm{d}x\\
                     &\,
                     + \int_{\mathds{R}^{N}} |\nabla ( \phi^{r} u)|^2 \,  \mathrm{d}x.\end{aligned}
             \end{equation}
             By the Poincar\'e inequality,
             \begin{equation*}
                 \int_{\mathds{R}^{N}} |\nabla( \phi^{r} u)|^2 \,  \mathrm{d}x\geqslant \mu_{r} \int_{ B_{2r}}  |\phi^{r} u|^2\, 
                 \mathrm{d}x \geqslant \mu_{r} \int_{B_{ r  }}  |\phi^{r} u|^2\,  \mathrm{d}x
                 \geqslant \mu_{r} \int_{B_{ r  }}  |u|^2\,  \mathrm{d}x.
             \end{equation*}
             Finally, from \eqref{est-2222}, we obtain
             \begin{equation*}
                 \begin{split}
                    2\|u\|_{\sigma}^{2} &\geqslant
                    \mu_{r}\left(1 -\frac{1}{r}\right)\int_{B_{ r  }}  |u|^2\,  \mathrm{d}x
                    +2\left(1-  \frac{1}{r^{2}}\right)\int_{B_{2r}\setminus B_{r }}  |u|^2\,  \mathrm{d}x
                    +2\int_{\mathds{R}^{N}\setminus B_{2r}}  |u|^2\,  \mathrm{d}x \\
                    &\geqslant C(r)\int_{\mathds{R}^{N}}|u|^{2} \, \mathrm{d}x,
                 \end{split}
             \end{equation*}
              where $C(r)=\min\left\{\mu_{r}\left(1 -\frac{1}{r}\right), 2\left(1-  \frac{1}{r^{2}}
              \right)\right\}$. By choosing  $r>1$, we achieve that $C(r)>0$.
              Thus, \eqref{ineq-1111} is proved.
              \vspace{2mm}
        \item
           In this step, we show that $\sigma(x)$ satisfies condition \eqref{coecivo} for all $s \in [2, 2^*)$.
           Let $M > 0$ be a number and $\Omega_M$ be defined as in Lemma
           \ref{Th1.4-sirakov}.
           Take a sequence $\{z_n\} \subset \mathds{R}^{N}$, $z_n\to\infty$, and $R>0$.  Let us show that
           \begin{align}
           \label{2show-1111}
           \lim_{n\to\infty} |\Omega_M \cap B_R(z_n)| = 0.
            \end{align}
           By Lemma \ref{Th1.4-sirakov}, this will imply
           condition \eqref{coecivo}.
           Note that if $M < 1$, then $\Omega_M \subset B_r(0)$.
           On the other hand, for sufficiently large $n$, $B_r(0)\cap B_R(z_n) = \varnothing$, and therefore, 
          $|\Omega_M \cap B_R(z_n)|= 0$.

            Thus, to prove \eqref{2show-1111}, it suffices to consider the case $M\geqslant 1$.
       By  the definition of $\sigma$, 
            \begin{equation*}
            \begin{split}
                \Omega_{M}
                =\left\{ (x,y)\in \mathds{R}^N\setminus B_r(0): |y|\leqslant \frac{\sqrt{M-1}}{\kappa(|x|+1)}  \right\}\cup
                \big(B_r(0)\cap \Omega_M\big).
            \end{split}
            \end{equation*}
 Since, for sufficiently large $n$, $B_r(0)\cap B_R(z_n) = \varnothing $,
 it suffices to prove that $\lim_{n\to\infty}|\Omega_{M,n}| = 0$, where
 \begin{align*}
 \Omega_{M,n} := \left\{ (x,y)\in \mathds{R}^N: |y|\leqslant \frac{\sqrt{M-1}}{\kappa(|x|+1)}  \right\} \cap B_R(z_n).
\end{align*}
 Let the sequences $\{x_n\}\subset \mathbb{R}$ and $\{y_n\}\subset \mathbb{R}^{N-1}$
 be such that $z_n=(x_n,y_n)$.
     Since $z_n\to \infty$, at least one of the sequences $|y_n|$ or $|x_n|$ tends to $\infty$.
           We claim that if $|y_n| \rightarrow \infty$, then there exists $n_0\in \mathbb{N}$ such that $\Omega_{M,n} = \varnothing $ for $n\geqslant n_0$.
           Indeed, representing $\mathbb{R}^N$ as $\mathbb{R}\times \mathbb{R}^{N-1}$,
           we note that $\Omega_M\cap \mathbb{R}^{N-1}$ is bounded while
           any point of $B(z_n, R)\cap\mathbb{R}^{N-1}$ goes to infinity.
           This proves the claim.     
           Now let  $|x_n|$ tend to infinity.
           Note that if $|x-x_n|\leqslant R$,
           then $|x|\geqslant |x_n|-R$.
           Therefore, when $|x_n|> R$,
           \begin{equation*}
               \begin{split}
                   \Omega_{M,n}   \subset & \,\, \Omega_{M,n} \cap\left( (x_n-R,x_n+R) \times \mathds{R}^{N-1}\right)\\
                    \subset & \left\{  (x,y)\in \mathds{R}^N:
                    |x-x_n| \leqslant R \text{ and } |y| \leqslant
                    \frac{\sqrt{M-1}}{\kappa(|x_n |-R+1)} \right\}.
               \end{split}
           \end{equation*}
           This implies that $\Omega_{M,n}$ is enclosed in a
           parallepiped whose $N-1$ sides have length $\frac{2\sqrt{M-1}}{\kappa(|x_n|-R+1)}$ and the other side is of constant length $2R$.  Hence,
           \begin{equation*}
               \left|  \Omega_{M,n} \right|\leqslant 2^{N}R  \left(\frac{\sqrt{M-1}}{\kappa(|x_n |-R+1)}
               \right)^{\!\!{N-1}} \rightarrow 0 \quad\text{as } n\rightarrow\infty
           \end{equation*}
           which proves \eqref{2show-1111}.
\end{enumerate}
\end{example}

       \begin{example}
\rm
Example \ref{Example_Sigma} shows that the class
$\Upsilon_{s}(\mathds{R}^{N})$, $s\in[2,2^*)$,
is non-empty, and moreover, contains function that take
negative values and are non-coercive. Here we give an example
of two representatives, $a$ and $b$, of $\Upsilon_{s}(\mathds{R}^{N})$
such that $a$ and $b$ can take negative values, are non-coercive, 
$E_a = E_b$, and the norms $\|\cdot\|_{a+}$ and
$\|\cdot\|_{b+}$ generate the same topology on $E_b$.

As before, let $s\in[2,2^{*})$.
Let $\mathcal A_{r,\kappa}$, $r,\kappa>0$,  denote the class
of functions $\sigma$ of the form \eqref{class-sigma}, 
depending on the parameters $r,\kappa>0$ and such that
 different functions from $\mathcal A_{r,\kappa}$
can take different values of in the ring $B_{\frac{r}2,r}$.
   Choose two functions  $a\in \mathcal A_{r_1,\kappa_1}$
   and $b\in \mathcal A_{r_2,\kappa_2}$
   that take negative values and are non-coercive.
   Let us show that $E_a=E_b$ and the norms
$\|\cdot \|_{a+}$ and $\|\cdot \|_{b+}$ are equivalent. Let $\varrho:=r_1+r_2$,
$a_1: = a\mathds{1}_{B_\varrho(0)}$,
$a_2: = a\mathds{1}_{\mathbb{R}^N\setminus B_\varrho(0)}$,
$b_1: = b\mathds{1}_{B_\varrho(0)}$,
$b_2: = b\mathds{1}_{\mathbb{R}^N\setminus B_\varrho(0)}$.
Note that $E_{a_2} = E_{b_2}$ and the norms
$\|\cdot\|_{a_2}$ and $\|\cdot\|_{b_2}$ are equivalent.
Indeed, suppose, for the sake of certainty, that $\kappa_1\leqslant \kappa_2$. Then, $\|\cdot\|_{a_2}\leqslant \|\cdot\|_{b_2} \leqslant \frac{\kappa_2}{\kappa_1}\|\cdot\|_{a_2}$.  Furthermore, since
$a_2 = (a_2)_+ \leqslant a_+$ and $b_2 = (b_2)_+ \leqslant b_+$, then
for all $u\in E_b$ and $v\in E_a$,
\begin{align}
\label{a2b2-2222}
\|u\|_{a_2} \leqslant \|u\|_{b_+} \quad \text{and} \quad
\|v\|_{b_2} \leqslant \frac{\kappa_2}{\kappa_1} \|v\|_{a_+}.
\end{align}
On the other hand,
$E_b \hookrightarrow H^1(\mathbb{R}^N) \hookrightarrow E_{(a_1)_+}$,
where both canonical embeddings are continuous. The continuity
of the first embedding is due to Lemma \ref{Embedded};
the second embedding and its continuity are implied
by the  boundedness of $a_+$ on $B_\varrho(0)$.
Therefore, there exists a constant $C_1>0$
such that $\|\cdot\|_{(a_1)_+} \leqslant C_1 \|\cdot\|_{b_+}$.
Together with \eqref{a2b2-2222}, this implies that for some constant
$\bar C_1>0$, 
\begin{align*}
\|u\|_{a_+} \leqslant \bar C_1 \|u\|_{b_+} \quad \text{for all} \;\;  u\in E_b.
\end{align*}
Likewise, we show that there exists a constant $C_2>0$ such that
for all $v\in E_a$,
\begin{align*}
\|v\|_{b_+} \leqslant C_2 \|v\|_{a_+}.
\end{align*}
Thus, $E_a=E_b$ and the norms  $\|\cdot\|_{a_+}$
and $\|\cdot\|_{b_+}$ are equivalent.
       \end{example}

Surprisingly, even strict positivity of the function $\sigma$ does not guarantee that condition \eqref{primeiroAutovalor} holds, 
as demonstrated in the following example.

     \begin{example}
     \rm
     \label{example3333}
     Here we give an example of a strictly positive function
     $\sigma$ for which condition \eqref{primeiroAutovalor}
     is not satisfied. 
  Consider the function
            \begin{align}
            \label{gauss}
            \sigma_{gauss}(x) =
            \frac1{(2\pi)^\frac{N}2} \exp\left(
            -\frac{|x|^2}2\right)
\end{align}
            which is known as a Gaussian density
            with the property $\int_{\mathbb{R}^N} \sigma_{gauss}(x) dx = 1$.
            Let us show that $\tilde \lambda_1(\sigma_{gauss}) = 0$. Define $u(x) = |x|^{-\frac{N}{2}}$ if $|x| > 1$ and $u(x) = 1$ if  $|x| \leqslant 1$. 
            It is straightforward to verify that
            $u \in E_\sigma$, $|u|\leqslant 1$, and $u \notin L_2(\mathds{R}^N)$.
         Further define $v^r = \frac{\phi^r u}
         {\|\phi^r u\|_{L_2}}$,
         where $\phi_r$ is given by \eqref{phir1111}.
         We notice the following properties:
         \begin{align}
         \label{relations-1111}
         \|v^r\|_{L_2} = 1; \quad
         \|\phi^r u\|_{L_2} \to \infty \; (r\to\infty);
         \quad
         |\phi^r u(x)| \leqslant 1 \;\, \text{on} \;\, \mathbb{R}^N.
\end{align}
             By \eqref{eq2.E_subset_H}, there exists a constant $C>0$
         such that
            \begin{equation*}
            \|\nabla v^r\|_{\scriptscriptstyle L_2} =
                     \frac{\|\nabla (\phi^r u)\|_{L_2}}
                     {\|\phi^r u\|_{L_2}}
                     \leqslant
                    C\,\frac{\|\nabla u\|_{L_2}
                    +\|u\|_{\scriptscriptstyle L_{2^*}}}{\|\phi^r u\|_{L_2}}
                    \to 0, \qquad r\to \infty.
    	   \end{equation*}
        Furthermore, by \eqref{relations-1111},  for every $x\in \mathbb{R}^N$,
        \begin{align*}
        |v^r(x)| \leqslant
        \|\phi_r u\|^{-1}_{\scriptscriptstyle L_2}
        \to 0 \quad \text{as}\; \; r\to \infty.
\end{align*}
        Since $v^r\in E_\sigma$ and $\int_{\mathbb{R}^N} \sigma_{gauss}(x) dx = 1$,
            $$
            \tilde \lambda_1(\sigma_{gauss}) \leqslant
            \|\nabla v^r\|^2_{L_2} + \int_{\mathbb{R}^N} \sigma(x) |v^r|^2 \mathrm{d}x \leqslant \|\nabla v^r\|^2_{L_2}
            + \|\phi_r u\|_{L_2}^{-2}
           \, \to 0, \qquad r\to \infty.
            $$
    \end{example}

    \section{Variational framework}
\label{Variational framework}
    \subsection{The concept of solution to system  \eqref{P}} 
    Here we define a weak solution to system \eqref{P}.
\begin{definition}
\label{def-weak-solution-1111}
\rm
Let $a,b\in L^\infty_{\rm loc}(\mathbb{R}^N)$.
    The pair of functions $(u, v) \in E_b \times E_b$ is called a \emph{weak solution}
    to system \eqref{P} if for all \(\varphi, \psi \in {\rm C}_{0}^{\infty}(\mathds{R}^N)\),
    \begin{equation}\begin{aligned}
    \label{23-1111}
         &\int_{\mathds{R}^{N}} \nabla u \nabla \varphi \, \mathrm{d}x+\int_{\mathds{R}^{N}} a(x) v \varphi \, \mathrm{d}x=\int_{\mathds{R}^{N}} f(x,u) 
         \varphi  \, \mathrm{d}x \\
    &\text{\hspace{-11mm} and} &&\\
    &\int_{\mathds{R}^{N}} \nabla v \nabla \psi \, \mathrm{d}x+\int_{\mathds{R}^{N}} b(x) v \psi \, \mathrm{d}x=\beta\int_{\mathds{R}^{N}} a(x) u \psi \, \mathrm{d}x.\end{aligned}
\end{equation}
          \end{definition}
     For the proof of Lemma \ref{type12-1111}, we invoke Proposition \ref{Ne} and Corollary \ref{cor-3.11}, whose proofs are given later. 
We emphasize that the proofs of Proposition \ref{Ne} and Corollary \ref{cor-3.11} are completely independent of Lemma \ref{type12-1111}.

    \begin{lemma}
        \label{type12-1111}
        Let $b\in\mathcal T(\mathbb{R}^N)$ and $a(x)$ satisfies {\rm \ref{f0a}}.
        A pair of functions $(u,v)\in E_b\times E_b$ is a weak solution to system \eqref{P} if and only if
        equations \eqref{23-1111} are true for all $\varphi,\psi\in E_b$.
    \end{lemma}
    \begin{proof}
        By Lemma \ref{dense}, ${\rm C}_{0}^{\infty}(\mathds{R}^{N})$
        is dense in $E_b$. Let $\varphi,\psi\in E_b$ and let $\varphi_n,\psi_n \in {\rm C}_{0}^{\infty}(\mathds{R}^{N})$
        be sequences converging to $\varphi$ and $\psi$, respectively, in $E_b$. The expressions on the left-hand sides of \eqref{23-1111}
        represent scalar products $\<u,\varphi\rangle_a$ and $\<v,\psi\rangle_b$, respectively, which
        can be approximated by $\<u,\varphi_n\rangle_a$ and $\<v,\psi_n\rangle_b$.  Furthermore, by  Proposition \ref{Ne} and Corollary \ref{cor-3.11},
        there exists a constant $C>0$ such that for all $u,w\in E_b$
        \begin{align}\label{fuw-1111}
            \left| \int_{\mathds{R}^{N}} f(x,u)  \varphi(x) \mathrm{d}x\right|\leqslant C \|u\|^p_{b}  \|\varphi\|_{b}.
\end{align}
        Finally, by Lemmas \ref{Embedded} and \ref{Lemaeqnorms} and the equivalence of the norms
            $\|\cdot\|_a$ and $\|\cdot\|_b$, there exists a constant $C_1>0$ such that
         \begin{align}
            \label{auw-1111}
            \left| \int_{\mathds{R}^{N}} a(x) u(x) \psi(x) \mathrm{d}x\right|
            &\,\leqslant  \left|\int_{\mathds{R}^{N}} a_+(x) u(x) \psi(x) \mathrm{d}x\right|
            + \esssup_{\mathbb{R}^N}|a_-| \|u\|_{L_2}\|\psi\|_{L_2}
            \leqslant C_1 \|u\|_b \|\psi\|_b.
\end{align}
        Estimates \eqref{fuw-1111} and \eqref{auw-1111}  imply that the integrals  $\int_{\mathds{R}^{N}} f(x,u)  \varphi \, \mathrm{d}x$ 
        and $\int_{\mathds{R}^{N}} a(x) u \psi \, \mathrm{d}x$ can be approximated by $\int_{\mathds{R}^{N}} f(x,u) \varphi_n \, \mathrm{d}x$ and
            $\int_{\mathds{R}^{N}} a(x) u \psi_n \, \mathrm{d}x$, respectively.
    \end{proof}

  \subsection{Equivalence of system \eqref{P} to a non-local equation} 
The following proposition provides a reformulation of system \eqref{P} in the form of 
one non-local equation that will be used throughout the remainder of this paper.
     \begin{proposition}\label{existen}  Assume $b\in\mathcal T(\mathbb{R}^N)$ 
     and $a(x)$ satisfies {\rm \ref{f0a}}.
         Let $u \in E_{b}$. Then, the equation
        \begin{equation}\label {P2}
        -\Delta v+b(x)v=\beta a(x)u 
        \end{equation}
        has a unique weak solution $v \in E_{b}$, that is,
        there exists a continuous solution operator $S_b: E_b \to E_b$,
        \begin{align*}
        S_{b} w: = \beta (-\Delta+b)^{-1} (a w)
\end{align*}
        such that equation \eqref{P2} is equivalent to
        \begin{equation}\label {DefSb}
        v=S_{b} u.
        \end{equation}
        Moreover,
        \begin{equation}\label {riesz}
    \|S_{b}u\|_{b}=\beta\sup_{\|w\|_b\leqslant1}\left|\int_{\mathds{R}^N} a(x) u w \, \mathrm{d} x\right|.
    \end{equation}
 \end{proposition}

    \begin{proof}
    Since the right-hand side of the second equation in \eqref{23-1111} represents the scalar product
    $\<v,\psi\rangle_b$, and in view of
  Lemma \ref{type12-1111}, we have to prove that
    there exists a unique element $v\in E_b$ such that 
    \begin{align}
    \label{scl-1111}
    \<v,w\rangle_{b} = \beta\int_{\mathds{R}^{N}} a(x) u w \, \mathrm{d}x \qquad \forall \;\; w\in E_b.
\end{align}
    Note that by \eqref{auw-1111},
    \begin{align}
    \label{func-1111}
    w \mapsto \beta\int_{\mathds{R}^{N}} a(x) u w \, \mathrm{d}x
\end{align}
    is a linear continuous functional on $E_b$. 
      By the Riesz representation theorem in the Hilbert space
    $E_b$, there exists a unique element
    $v\in E_b$ such that \eqref{scl-1111} is fulfilled.
    Thus, we have defined a map
    \begin{align*}
    S_b: E_b\to E_b, \quad u\mapsto v.
\end{align*}
    Identity \eqref{riesz} is implied by the continuity of the functional
    \eqref{func-1111}. The continuity of $S_b: E_b \to E_b$ follows from
    \eqref{auw-1111}.
     \end{proof}
     \begin{corollary}
     \label{lap-inv-1111}
      Assume that $a(x)$ satisfies  {\rm \ref{f0a}} and $\sigma(x)$ also
      satisfies {\rm \ref{f0a}} in place of $a(x)$. 
         Let $u \in E_{b}$. Then, each of the equations
        \begin{equation}\label {Psg}
        -\Delta v+\sigma(x)v= u \quad \text{and} \quad -\Delta v+ \sigma(x)v=a(x) u  
        \end{equation}
        has a unique weak solution, that is,
        there exist continuous solution operators $(-\Delta+\sigma)^{-1}: E_b \to E_b$
        and $(-\Delta+\sigma)^{-1}a: E_b \to E_b$
        such that equations \eqref{Psg} are,  respectively, equivalent to
        \begin{equation*}
        v= (-\Delta+\sigma)^{-1}  u \quad \text{and} \quad v= (-\Delta+\sigma)^{-1} (a u).
              \end{equation*}          
  \end{corollary}
  \begin{proof}
   Note that by  Lemma \ref{Embedded}, the argument used in the proof also works for $a=1$, so
   the operator $(-\Delta+b)^{-1}: E_b \to E_b$ exists. Next, if $\sigma \in \mathcal T(\mathbb{R}^N)$, by Lemma \ref{sg-norm-1111},
   $\langle\cdot,\cdot\rangle_\sigma$ is a scalar product on $E_\sigma = E_b$. Otherwise, 
   if $\sigma\geqslant 0$ (a.e.) is measurable with $\sigma\ne 0$ a.e.,
   $\langle\cdot,\cdot\rangle_\sigma$ is also a scalar product on $E_\sigma$. Since \eqref{func-1111} is a linear continuous functional 
   on $E_b$, by \ref{f0a}, it is also a linear continuous functional 
   on $E_\sigma$. Now the argument at the end of the proof of Propositin \ref{existen}, 
   involving the Riesz theorem, implies the statement of the corollary.
  \end{proof}
       Thus, with the help of the operator $S_{b}$, system \eqref{P} is  reduced to one non-local equation
\begin{equation}\label {PS}
\begin{aligned}
-\Delta u + a(x)S_{b}(u) = f(x,u).
\end{aligned}
\end{equation}
    Our goal now is to prove the existence of a weak solution
    to \eqref{PS} in the sense of \eqref{23-1111}. In view of Lemma 
    \ref{type12-1111}, we have to show that
    there exists a function $u\in E_b$ such that
    \begin{equation}
    \label{weak-solution-2222}
        \int_{\mathds{R}^{N}} \nabla u \nabla w \, \mathrm{d}x+\int_{\mathds{R}^{N}} a(x) (S_{b}u) w \, 
        \mathrm{d}x =\int_{\mathds{R}^{N}} f(x,u) w   \qquad \forall w \in E_{b}.
    \end{equation}
    \subsection{Properties of the operator $S_{b}$}
    \label {prps-1111}
    \begin{lemma} \label {Properties}
       Assume that $b\in \mathcal T(\mathbb{R}^N)$. Then,
        the operator $ S_{b}:(E_{b},\|\cdot \|_b )\to (E_{b},\|\cdot \|_b ) $ 
        satisfies the following symmetry property:
        \begin{equation}\label {simetriaSb}
        \int_{\mathds{R}^{N}} a(x) u(S_{b}v) \, \mathrm{d}x = \int_{\mathds{R}^{N}} a(x) v(S_{b}u) \, \mathrm{d}x \quad \text{for all }u,v\in E_{b}.
        \end{equation}
        \end{lemma}
    \begin{proof}
       We note that by \eqref{DefSb}
       and \eqref{scl-1111}, for all $u,w\in E_{b}$,
        \begin{equation}\label {au}
            \langle S_{b}u,w \rangle_{b}=\beta\int_{\mathds{R}^{N}}a(x)uw \,
            \mathrm{d}x.
        \end{equation}
        By writing \eqref{au} with $w = S_bv$, by the definition of the scalar product
        $\langle\cdot,\cdot\rangle_b$, from \eqref{au}
        we obtain
        \begin{equation}
        \label{scl-prd-1111}
                \int_{\mathds{R}^{N}} a(x) uS_{b}v \, \mathrm{d}x  = \frac{1}{\beta} \left( \int_{\mathds{R}^{N}}
                \big(\nabla (S_{b}u), \nabla (S_{b}v)\big) \, \mathrm{d}x +\int_{\mathds{R}^{N}}b(x)S_{b}u S_{b}v \, \mathrm{d}x\right).
        \end{equation}
        Note that the right-hand side of \eqref{scl-prd-1111} is symmetric with respect to $u$ and $v$. Changing the roles
        of $u$ and $v$ on the left-hand side, we obtain that
        the integral $\int_{\mathds{R}^{N}} a(x) vS_{b}u \, \mathrm{d}x$
        also equals the right-hand side of \eqref{scl-prd-1111}.
        This proves \eqref{simetriaSb}.
    \end{proof}

\subsection{The Hilbert space $(E_b, \|\cdot\|_{ab})$}
\label {prps-2222}
    In this subsection, we introduce a norm $\|\cdot\|_{ab}$ on $E_b$ equivalent to the norm $\|\cdot\|_b$.
    More specifically, we will show that the bilinear form
    \begin{equation}
    \label{inner-ab}
        \langle u,v \rangle_{ab}: =\int_{\mathds{R}^{N}} \nabla u \nabla v + 
        a(x) u S_{b}(v) \, \mathrm{d}x
    \end{equation}
    defines an inner product on $E_b$. 
     \begin{proposition}\label{Ne}
     Let $b\in \mathcal T(\mathbb{R}^N)$ and $a(x)$ satisfy {\rm \ref{f0a}}.
    Then,
       $(E_{b},\langle \cdot , \cdot\rangle_{ab})$ is
       a Hilbert space.
       Furthermore,
       \begin{align}
       \label{ab-exp-1111}
       \|u\|^2_{ab} =
        \|\nabla u\|^2_{L_2(\mathbb{R}^N)}  +\beta^{-1}\|S_b u\|_{b}^2
\end{align}
       and the norms $\|\cdot \|_{ab}$ and $\|\cdot \|_{b}$ are equivalent.
    \end{proposition}
    \begin{proof}
   By Lemma \ref{Properties}, 
   \eqref{inner-ab} defines a symmetric bilinear form on $E_b\times E_b$.
   To show that $\langle \cdot, \cdot \rangle_{ab}$ is an inner product, it suffices to verify that
   \begin{align}
   \label{ab-scl-1111}
   \langle u, u \rangle_{ab} \geqslant 0 \quad \text{for all}\;\;
   u \in E_b
\end{align}
 and that $\langle u, u \rangle_{ab} = 0$ if and only if $u = 0$.
    By \eqref{scl-prd-1111}, for all $u\in E_b$,
    \begin{equation}
    \label{norm-a-1111}
                \int_{\mathds{R}^{N}} a(x) uS_{b}u \, \mathrm{d}x  = \beta^{-1} \left( \int_{\mathds{R}^{N}}
                |\nabla S_{b}u|^2  \, \mathrm{d}x +\int_{\mathds{R}^{N}}b(x)(S_{b}u)^2  \, \mathrm{d}x\right)= \beta^{-1}\|S_bu\|_{b}^2.
        \end{equation}
    Therefore,  by the definition of $\langle \cdot, \cdot \rangle_{ab}$,
        \begin{equation}
            \label {Sb2}
                \langle u, u \rangle_{ab} =
                \int_{\mathds{R}^{N}}| \nabla u|^2\, \mathrm{d}x +   \int_{\mathds{R}^{N}} a(x)  S_{b}(u) u \, \mathrm{d}x
                = \|\nabla u\|^2_{L_2(\mathbb{R}^N)}  +\beta^{-1}\|S_b u\|_{b}^2 \geqslant 0.
        \end{equation}
        This proves \eqref{ab-exp-1111} and \eqref{ab-scl-1111}.
 Further, $\langle u, u \rangle_{ab} = 0$ if and only if
 $\|\nabla u\|_{L_2(\mathbb{R}^N)}=0$ and $\|S_b u\|_{b}=0$.
 This implies that $\nabla u = 0$ and $S_b u = 0$ a.e.
 Consequently, $u=const$ a.e.
 Suppose $u=const \ne 0$. Then, from \eqref{au}, we obtain that
 for all $w\in {{\rm C}}^\infty_0(\mathbb{R}^N)$,
 $\int_{\mathbb{R}^N} a(x) w(x) {\rm d}x = 0$. Hence $\rho_\varepsilon \ast a = 0$ on $\mathbb{R}^N$, where
$\rho_\varepsilon$ is the standard mollifier.
  Since $\lim_{\varepsilon\to 0} \rho_\varepsilon \ast a = a$ a.e.~on $L_2(B_R(0))$ for any ball $B_R(0)$, we obtain that 
  $a = 0$ a.e.~on $\mathbb{R}^N$.
 Hence, $u=const = 0$ a.e. Thus, we have
 proved that $\langle u,v \rangle_{ab}$ is an inner product
 on $E_b$.

        Let us show now the equivalence of
        the norms $\|\cdot \|_{ab}$ and $\|\cdot \|_{b}$.
        By the boundedness of $S_b$ and \eqref{ab-exp-1111}, 
        there exists a constant $C>0$ such that
        \begin{equation}
        \label{equiv-2222}
            \|u\|_{ab}\leqslant C \|u\|_{b}.
        \end{equation}
        On the other hand, by \eqref{au}, for every $u\in E_b$,
         \begin{align*}
        \left|\int a(x) u^2 dx\right| \leqslant
          \beta^{-1} \|S_b u\|_b \|u\|_b.
\end{align*}
        Therefore, by the equivalence of $\|\cdot\|_a$ and
        $\|\cdot\|_b$,
        there exist constants $C_1,C_2>0$ such that
        \begin{align*}
        \|u\|^2_a \leqslant &\, \|\nabla u\|_{L_2}^2 +
        \left|\int a(x) u^2 dx\right| \leqslant
        \|\nabla u\|_{L_2}^2  +\beta^{-1}\|S_b u\|_b \|u\|_b\\
        \leqslant & \, \|\nabla u\|_{L_2}^2
        +C_1 \beta^{-1} \|S_b u\|_b \|u\|_a
        \leqslant  C_2\big(\|\nabla u\|_{L_2}^2
        +\beta^{-1} \|S_b u\|_b^2\Big) + \frac{\|u\|^2_a}2.
\end{align*}
        Using again the equivalence of $\|\cdot\|_a$ and
        $\|\cdot\|_b$ along with \eqref{ab-exp-1111}, we conclude that
        there exists a constant $C_3>0$ such that
        \begin{align*}
        \|u\|_b \leqslant C_3 \|u\|_{ab}.
\end{align*}
         Together with \eqref{equiv-2222},
        this proves the equivalence of the norms
        $\|\cdot \|_{ab}$ and $\|\cdot \|_{b}$. In particular,
        the space $(E_b, \|\cdot\|_{ab})$ is complete, and hence,
        a Hilbert space.
\end{proof}
      \begin{proposition}\label{compa}
      Let $b\in \Upsilon_{s}(\mathds{R}^{N})$ for some
      $s\in (2,2^{*})$ and $a(x)$ satisfy {\rm \ref{f0a}}.
    Then, the canonical embedding 
         \begin{align}
         \label{Ls-embd=1111}
         (E_{b}, \|\cdot \|_{ab}) \hookrightarrow L_{s}(\mathds{R}^N)
\end{align}
        is compact for all $s\in (2,2^{*})$. If $b\in \Upsilon_{2}(\mathds{R}^{N})$,   
        then \eqref{Ls-embd=1111} is compact for all $s\in [2,2^{*})$.
    \end{proposition}
    The following lemma will be used in the proof of Proposition \ref{compa}.
    \begin{lemma}\label {LemmaSirakov}
        Let  $s\in [2,2^{*})$ and let $b\in \Upsilon_{s}(\mathds{R}^{N})$. Then, for any sequence   
        $\left\{x_n\right\} \subset \mathds{R}^{N}$ such that $\lim_{n\to\infty} x_n = \infty$, and for any $r>0$,
     \begin{equation*}
          \lim _{n \rightarrow \infty} \nu_s\left(b, B_{r}(x_{n})\right)=\infty,
      \end{equation*}
      if only if
        \begin{equation*}
            \lim_{R\rightarrow \infty}\nu_{s}(b,\mathds{R}^{N}\setminus  \overline{B_R(0)})=
            \infty.
        \end{equation*}
           \end{lemma}
    \begin{proof}
    For the proof, see \cite[Proposition 2.1]{MR1782990}.
        \end{proof}
    \begin{proof}[Proof of Proposition \ref{compa}]
        By Remark \ref{rm2.6-1111}, the canonical embedding
        $E_{b} \hookrightarrow L_s(\mathds{R}^N)$, $s \in [2, 2^{*})$, is continuous.
        Note that we can show the compactness of this embedding with respect
        to any of the equivalent norms $\|\cdot\|_{ab}$, $\|\cdot\|_a$, or $\|\cdot\|_{a_+}$.
        The equivalence of the first two norms follows from Proposition \ref{Ne};
        the equivalence of the second and the third norms is due to Lemma \ref{Lemaeqnorms}.
        To show that the canonical embedding is compact,
        take a sequence $\{u_n\}\subset E_b$ such that $\|u_n\|^2_{a_+}<C$
        for some constant $C>0$. We have
        to prove that there exists a subsequence $\{u_{n_k}\}
        \subset \{u_n\} $ that converges in $L_{s}(\mathds{R}^N)$.
        Note that $\{u_n\}$ contains a weakly convergent subsequence $u_{n_k}\rightharpoonup u_0\in E_b$.
        Without loss of generality, subtracting if necessary
        the limit $u_0$ from $u_{n_k}$, and also for notational
        simplicity, we assume that
         $u_n\rightharpoonup 0$ in $E_b$.
         Let $\phi^r$ be the cut-off function defined by  \eqref{phir1111}.
         We have
        \begin{equation}\label {compa1}
            \|u_{n}\|_{L_s(\mathbb{R}^N)} \leqslant
            \|\phi^{r} u_{n}\|_{L_s(\mathbb{R}^N)} +
            \|(1-\phi^{r}) u_n\|_{L_s(\mathbb{R}^N)}.
        \end{equation}
        Since $0\leqslant 1-\phi^r \leqslant \mathds{1}_{\mathbb{R}^N\setminus B_r}$,
        by the definition of $\nu_s(\mathbb{R}^N\setminus B_r)$, see \eqref{nus},
        and the boundedness of $\|u_n\|_{a_+}$, we have
        \begin{align*}
        \|(1-\phi^{r}) u_n\|^2_{L_s(\mathbb{R}^N)} \leqslant
        \|u_n\|^2_{L_s(\mathbb{R}^N\setminus B_r)}
        \leqslant \frac{\int_{\mathbb{R}^N\setminus B_r}
        \big(|\nabla u_n|^2 + a_+|u_n|^2\big) {\rm d}x}
        {\nu_s(a_+,\mathbb{R}^N\setminus B_r)}
        \leqslant \frac{C}{\nu_s(a,\mathbb{R}^N\setminus B_r)}.
\end{align*}
        By Lemma \ref{LemmaSirakov} and Remark \ref{usp-s-1111}, if $b\in \Upsilon_t(\mathds{R}^N)$ for some $t\in (2,2^*)$, then
        $\lim_{r\to\infty}\nu_s(a,\mathbb{R}^N\setminus B_r) = \infty$ for all $s\in (2,2^*)$; and if $b\in \Upsilon_2(\mathds{R}^N)$,
        then 
        $\lim_{r\to\infty}\nu_s(a,\mathbb{R}^N\setminus B_r) = \infty$ for all  $s\in [2,2^*)$. We continue the proof with some
        $s\in [2,2^*)$.
        Let $\varepsilon>0$ be fixed arbitrarily, and let
    $r>0$ be such that
        \begin{align*}
        \|(1-\phi^{r}) u_n\|_{L_s(\mathbb{R}^N)} < \varepsilon.
\end{align*}
        Remark that the linear operator $E_b\to E_b$,
        $u\mapsto \phi^r u$ is bounded. Indeed,
        since $|\phi^r|\leqslant 1$ and $|\nabla \phi^r|\leqslant \frac1{r}$, there exists a constant $C_r>0$ such that
        \begin{align*}
        \|\phi^r u\|^2_{a_+} =
        \int_{\mathbb{R}^N} |\nabla (\phi^r u)|^2 dx
        + \int_{\mathbb{R}^N} a_+(x) (\phi^r u)^2 dx
        \leqslant \|u\|^2_{a_+} + r^{-2} \|u\|^2_{L_2(\mathbb{R}^N)}
        \leqslant  C_r \|u\|^2_{a_+}.
\end{align*}
        Above, we have used the continuity of the canonical embedding $E_{b} \hookrightarrow L_2(\mathds{R}^N)$.
        Consequently,
        $\phi^r u_n\rightharpoonup 0$ in $E_b$.
        But $\phi^r u_n$ is in  $H^{1}_{0}(B_{2r})$ and
        the canonical embedding
        $H^{1}_{0}(B_{2r}) \hookrightarrow L_{s}(B_{2r})$
      is compact, see \cite[Theorem 9.16]{zbMATH05633610}. Therefore,
        $\{\phi^r u_n\}$ contains a subsequence $\{\phi^r u_{n_k}\}$ such that
        $\lim_{k\to\infty}\phi^r u_{n_k} = 0$ in $L_s(\mathbb{R}^N)$.
        Passing to the limit  in \eqref{compa1} as $k\to \infty$, we obtain that
        \begin{align*}
        \limsup_{k\to \infty}\|u_{n_k}\|_{L_s(\mathbb{R}^N)}
        \leqslant \varepsilon.
\end{align*}
        Since $\varepsilon>0$ is arbitrary, the result follows.
           \end{proof}

\subsection{The energy functional}
Here we introduce the energy functional associated with equation
\eqref{PS}. Define the map
\begin{equation}\label {psi1111}
        \Psi: E_b \rightarrow \mathds{R}, \quad
        \Psi(u) =  \int_{\mathds{R}^{N}} F(x,u(x)) \mathrm{d}x;
        \quad \text{with} \;\;  F(x,u): = \int_0^u f(x,v) dv.
    \end{equation}
    \begin{lemma}
    \label{frch-2222}
    The Fr\'echet derivative $D\Psi$ of $\Psi: E_b\to \mathbb{R}$ along
    the space $(E_b, \|\cdot\|_{ab})$ equals
    \begin{align*}
    D\Psi(u) h = \int_{\mathds{R}^{N}} f(x,u(x)) h(x) \mathrm{d}x,
    \quad u,h\in E_b.
\end{align*}
    \end{lemma}
    For the proof of Lemma \ref{frch-2222}, we need the following lemma.
    \begin{lemma}\label {A12}
    Let $b\in \Upsilon_s(\mathds{R}^N)$ for some $s\in (2,2^*)$. Further let
        the function $f$ satisfy conditions {\rm \ref{f1}} and {\rm \ref{f2}}
        and $a(x)$  satisfy {\rm \ref{f0a}}. 
              Then, there exist constants $C, \hat C>0$ and a number
              $s\in (2,2^*)$ such that   for all $u,v,w\in E_b$,
              \begin{align}
              \label{lip-1111}
              \Big|\int_{\mathds{R}^{N}}\hspace{-1mm} \big[f(x,u(x))- f(x,v(x))\big] w(x) \mathrm{d}x \Big|
               \leqslant C  \|u\|^{p-1}_{ab} \|u-v\|_{L_s} \|w\|_{ab} 
              \leqslant \hat C \|u\|^{p-1}_{ab} \|u-v\|_{ab}  \|w\|_{ab}.
\end{align}
    \end{lemma}
   \begin{proof}
    Let $u,v,w\in E_{b}$, $\|w\|_{ab} \leqslant 1$. By \ref{f2},
   \begin{align}
    \label {Auv}
         \Big|\int_{\mathds{R}^{N}} \big[f(x,u(x))- f(x,v(x))\big] w(x) \mathrm{d}x \Big|
         \leqslant& \int_{\mathds{R}^N}|f(x, u)-f(x, v)|\, |w| \, \mathrm{d}x \nonumber \\
        \leqslant& \,C_0\int_{\mathds{R}^N} \big(1+\phi(x)^{1/\alpha}\big)|u-v| \, |w|\, \mathrm{d} x+ \\
        + & \,C_0\int_{\mathds{R}^N} \big(1+\phi(x)^{1/\alpha}\big) (|u|^{p-1}+|v|^{p-1})|u-v| |w|\, \mathrm{d} x.\nonumber
\end{align}
We will start by bounding the second term on the right-hand side of \eqref{Auv}.
For this, we need to bound each of the terms
\begin{align}
\label{terms-1111}
\int_{\mathds{R}^N}|u|^{p-1}|u-v||w| \, \mathrm{d}x \quad \text{and}
\quad
\int_{\mathds{R}^N} \phi(x)^{\frac1\alpha}|u|^{p-1}|u-v||w| \, \mathrm{d}x.
\end{align}
First, we notice that by Lemmas \ref{Embedded}, \ref{Lemaeqnorms},
       Proposition \ref{Ne}, and Remark \ref{rm2.6-1111},
       for every $u\in E_b$,  $t\in [2,2^*)$, and $v\in L_t$,
       there exist constants $C_b>0$ and $C_t>0$ such that
       \begin{align*}
       \|u\|_{\phi} \leqslant C_b \|u\|_{ab} \quad \text{and} \quad
       \|v\|_{L_t} \leqslant C_t \|v\|_{ab}.
\end{align*}
       Until the end of the proof, we will use these bounds without further justification.
    
We start by bounding  the second expression in \eqref{terms-1111}
for the following interval of values of $p-1$:
    \begin{align}
    \label {p2}
     \frac4N \leqslant p-1 < \frac4{N-2}\Big(1-\frac1\alpha\Big).
\end{align}
    It is straightforward to verify that since $\alpha>\frac{N}2$, the
    above interval is non-empty.
  By the generalized H\"older inequality with the exponents  $q_1=\alpha$, 
$q_2$ to be determined later as it is dictated by the H\"older inequality,
$q_3= 2^*$, and $q_4 = (\frac1{2^*} + \varepsilon)^{-1}$,
where $\varepsilon\in (0,\frac12-\frac1{2^*})$ is to be fixed later,
    we obtain that there exist constants 
    $C_4, \hat C_4>0$ such that
    \begin{equation}\label {Kauvp}
        \begin{split}
            \int_{\mathds{R}^N}& {\phi}(x)^{\frac1\alpha}|u|^{p-1}|u-v||w| \, \mathrm{d}x
            = 
            \int_{\mathds{R}^N} \big({\phi}(x)^{\frac1\alpha}|u|^\frac2\alpha\big) |u|^{p-1-\frac2\alpha} |u-v||w| \, \mathrm{d}x\\
            &\leqslant \big\|{\phi}^{\frac1\alpha}\, |u|^{\frac2\alpha}\big\|_{\scriptscriptstyle L_\alpha} \,
            \||u|^{p-1-\frac2\alpha}\|_{\scriptscriptstyle L_{q_2}}
              \|w\|_{\scriptscriptstyle L_{2^*}} \|u-v\|_{L_{ (\frac1{2^*} + \varepsilon)^{-1}}}\\
             &\leqslant \|u\|_{\phi}^\frac2{\alpha} \|u\|_{L_{q_2(p-1-\frac2\alpha)}}^{p-1-\frac2\alpha}
              \|w\|_{\scriptscriptstyle L_{2^*}} \|u-v\|_{L_{ (\frac1{2^*} + \varepsilon)^{-1}}}\\
              &\leqslant C_4 \,\|u\|^{p-1}_{ab}
            \, \|u-v\|_{L_{ (\frac1{2^*} + \varepsilon)^{-1}}} \leqslant \hat C_4 \,\|u\|^{p-1}_{ab} \, \|u-v\|_{ab}.
        \end{split}
    \end{equation}
    We furthermore introduce $\varepsilon_i=\delta_i \varepsilon$, $i=1,2$, where
    $\delta_i>0$ are constants depending on $N,\alpha$, and $p$ that can be computed explicitly and whose computation we omit for notational simplicity. 
    The explicit computation shows that $\frac1{q_2}+\varepsilon = \frac{2\alpha-N}{\alpha N}$,
    and hence,
    \begin{align}
    \label{q2}
    q_2 = \frac{\alpha N}{2\alpha - N-\varepsilon_1}.
\end{align}
     It remains to remark that in order to bound the $L_{q_2(p-1-\frac2\alpha)}(\mathbb{R}^N)$-norm
    by the $\|\cdot\|_{ab}$-norm, it suffices to show that
 \begin{align}
 \label{q2-int}
 2 < q_2\big(p-1-\frac2\alpha\big)< 2^*
\end{align}
by Remark \ref{rm2.6-1111}.
 To show \eqref{q2-int}, we use the expression
  $N=\frac{2\cdot 2^*}{2^*-2}$ and equation \eqref{q2} to obtain 
    \begin{equation*}
        q_2\Big(p-1-\frac2\alpha\Big)=  2^*
        \frac{\alpha (p-1) -2}{(2^*-2)\alpha-2^*-\varepsilon_2}.
    \end{equation*}
    Since, by \eqref{p2},  $(2^*-2)\frac{2}{2^*} \leqslant p-1 < (2^*-2)\big(1-\frac1\alpha\big)$, 
    we obain that
    \begin{align*}
    2^*\frac{\alpha (p-1) -2}{(2^*-2)\alpha-2^*} < 2^*\frac{(2^*-2)(\alpha-1) - 2}{(2^*-2)\alpha-2^*} = 2^*.
\end{align*}
    Choose $\varepsilon_2>0$ sufficiently small so that
    \begin{align*}
    2^*\frac{\alpha (p-1) -2}{(2^*-2)\alpha-2^*-\varepsilon_2}<2^*.
\end{align*}
    On the other hand, 
    \begin{align*}
    2^*\frac{\alpha (p-1) -2}{\alpha (2^*-2)-2^*-\varepsilon_2} >
    \frac{2\alpha(2-2^*)  - 2 \cdot 2^*}{\alpha(2-2^*)-2^*} = 2.
\end{align*}
Hence, \eqref{q2-int} is proved.
    
    Now we consider the case $N\geqslant 4$ and 
    \begin{equation}
    \label{p-2222}
        \frac2{N-2}\Big(1-\frac2{\alpha}\Big) \leqslant p-1 < \frac{4}{N} = \frac{2(2^*-2)}{2^*}.
    \end{equation}
    We note that if $N=3$, the interval in \eqref{p-2222} is empty.
    Thus, the case $N=3$ is included in \eqref{p2}.

    As before, we will apply Lemma \ref{Embedded}, Proposition \ref{Ne},
     and Remark   \ref{rm2.6-1111} to bound the $\|\cdot\|_{a_+}$-norms
     and the $L_s$-type norms with $2< s < 2^*$. We have
    \begin{align}
\label{ES_RHS-1111}
      \int_{\mathds{R}^N}{\phi}(x)^{\frac1\alpha}|u|^{p-1}|u-v||w| \, \mathrm{d}x
=  \int_{\mathds{R}^N}\big({\phi}(x)^{\frac1\alpha} |w|^\frac2\alpha\big)
|u|^{p-1}|w|^{1-\frac2{\alpha}} |u-v|\, \mathrm{d}x.
\end{align}
      To expression \eqref{ES_RHS-1111}, we apply H\"older's inequality
      with the exponents $q_1=\alpha$, $q_2 = \frac{2^*}{p-1}$,
      $q_3 = 2^*(1 - \frac2\alpha)^{-1}$,  and $q_4$ 
      as it is dictated by the H\"older inequality.
      To make clear the validity of this application, we note that
      $q_2>\frac{2}{p-1}>\frac{N}2$, $q_3>2^*$.
      We obtain that the right-hand side of \eqref{ES_RHS-1111} can be
      bounded by 
      \begin{align}
      \label{C6-1111}
      \|w\|_{\phi}^\frac2\alpha \|u\|_{L_{2^*}}^{p-1}
      \|w\|_{L_{2^*}}^{1-\frac2\alpha} \|u-v\|_{L_{q_4}}
      \leqslant C_6 \|w\|_{ab} \|u\|^{p-1}_{ab} \|u-v\|_{L_{q_4}},
\end{align}
      where $C_6>0$ is a constant.
Let us show that $2 < q_4 < 2^*$. The same computation
is valid to justify the previous application of H\"older's inequality.
We have
\begin{align}
\frac1{q_4} = &\, 1 - \frac1\alpha -\frac{p-1}{2^*} -
\frac1{2^*}\Big(1-\frac2\alpha\Big)
= \frac1{2^*}+ \frac{2^*-2}{2^*} - \frac1\alpha
- \frac2{2^*} \Big(\frac{(p-1)}2 - \frac1\alpha\Big)\notag\\
> &\, \frac1{2^*} + \frac2{N} - \frac1\alpha
- \frac2{2^*} \Big(\frac2{N} - \frac1\alpha\Big)
= \frac1{2^*} + \Big(1- \frac2{2^*}\Big) \Big(\frac2{N} - \frac1\alpha\Big)
> \frac1{2^*}.
\label{bigger-est-1111}
\end{align}
On the other hand, the left inequality in \eqref{p-2222} can be rewritten
as follows:
\begin{align*}
p>1+ \frac{2}{N-2}\Big(1-\frac2\alpha\Big) = \frac{N}{N-2}
- \frac4{(N-2)\alpha} = \frac{2^*}2 - \frac{2\cdot 2^*}{\alpha N}.
\end{align*}
Using the above inequality and noticing that
$\frac{2^*-2}{2^*} = \frac2N$, we obtain
\begin{align*}
1-\frac1{q_4} =  \frac1\alpha +\frac{p-1}{2^*}
+ \frac1{2^*}\Big(1-\frac2\alpha\Big)
=\frac1\alpha \frac{2^*-2}{2^*} + \frac{p}{2^*} >
\frac2{\alpha N} + \frac12 - \frac2{\alpha N} = \frac12.
\end{align*}
This together with \eqref{bigger-est-1111} implies that
$2< q_4<2^*$. Therefore, \eqref{C6-1111} can be bounded by the expression
\begin{align}
\label{68-1111}
C_7  \|u\|^{p-1}_{ab} \|u-v\|_{L_{q_4}}
\leqslant \hat C_7  \|u\|^{p-1}_{ab} \|u-v\|_{ab},
\end{align}
where $C_7,\hat C_7>0$ are constants.

Now consider the case
\begin{equation}
    \label{p-3333}
        0 < p-1 < \frac2{N-2}\Big(1-\frac2{\alpha}\Big) \quad
        \text{and} \quad {\alpha\leqslant N}.
    \end{equation}
    Since $\alpha\leqslant N$, it is straightforward to verify that $p-1<\frac2N$.

Applying H\"older's inequality to the integral \eqref{ES_RHS-1111}
with the exponents $q_1=\alpha$, $q_2 = \frac{2}{p-1}$,
      $q_3 = 2(1 - \frac2\alpha)^{-1}$,  and $q_4 = \frac2{2-p}$,
         we obtain the following bound for the right-hand side of
      \eqref{ES_RHS-1111}
\begin{align}
\label{est-4444}
      \|w\|_{\phi}^\frac2\alpha \|u\|_{L_2}^{p-1}
      \|w\|_{L_2}^{1-\frac2\alpha} \|u-v\|_{L_{\frac2{2-p}}}
      \leqslant C_8 \|w\|_{ab} \|u\|^{p-1}_{ab} \|u-v\|_{L_{\frac2{2-p}}}.
\end{align}
      It is straightforward to notice that
$\frac2{2-p} = \frac2{1-(p-1)} > 2$. On the other hand, by \eqref{p-3333},
\begin{align*}
\frac2{2-p} = \frac2{1-(p-1)} < \frac2{1-\frac2N} = 2^*.
\end{align*}
Therefore, there exist constants $C_9, \hat C_9>0$ such that
the right-hand side of \eqref{est-4444} has the bound
\begin{align}
\label{68-2222}
C_9  \|u\|^{p-1}_{ab} \|u-v\|_{L_{\frac2{2-p}}} \leqslant 
C_9  \|u\|^{p-1}_{ab} \|u-v\|_{ab}.
\end{align}
The inequality \eqref{lip-1111} follows now by \eqref{Kauvp},
\eqref{68-1111}, and \eqref{68-2222}. 

It remains to consider the case
    \begin{align*}
    0 < p-1 \leqslant \frac{2}N \quad \text{and} \quad {\alpha > N}.
\end{align*}
The case $p-1 = \frac2{N}$, is included into \eqref{p-2222} if we set, in \eqref{p-2222},  $\alpha=N$.
Finally, case $p-1< \frac2{N}$ has been already considered.
The lemma is proved.
\end{proof}
    \begin{proof}[Proof of Lemma \ref{frch-2222}]
   Let $u,h\in E_b$. For all $x\in \mathbb{R}^N$, there exists $\theta(x)\in (0,1)$ such that
  \[
F(x, u + h) - F(x, u) = f\big(x, u + \theta h\big) h.
\]
By Lemma \ref{A12}, there exists a constant $C >0$ and a number
              $s\in (2,2^*)$ such that
\begin{multline*}
\frac{\Psi(u+h) - \Psi(u) - D\Psi(u) h}{\|h\|_{ab}}
=\frac{
\int_{\mathds{R}^n} \big(f(x, u + \theta h) - f(x, u)\big) h \, dx}{\|h\|_{ab}}
\leqslant C \|u\|^{p-1}_{ab} \|\theta h\|_{L_s}\\ \leqslant C \|u\|^{p-1}_{ab} \|h\|_{L_s}. 
\end{multline*}
By Remark \ref{rm2.6-1111},  there exists a constant $\hat C>0$ such that the right-hand side is smaller than
\begin{align*}
\hat C\|u\|^{p-1}_{ab} \|h\|_{ab}.
\end{align*}
This proves the lemma.
    \end{proof}
    \begin{corollary}
    \label{cor-3.11}
   Let the assumptions of Lemma \ref{A12} be fulfilled. Then, there exists a constant $C>0$ such that
   for all $u,w \in E_b$,
     \begin{align*}
                \Big|\int_{\mathds{R}^{N}} f(x,u(x)) w(x) \mathrm{d}x \Big|             
              \leqslant  C \|u\|^p_{ab}  \|w\|_{ab}.
\end{align*}
    \end{corollary}
    \begin{proof}
    It suffices to notice that $f(x,0) = 0$ and apply Lemma \ref{A12}.
    \end{proof}
    \begin{corollary}
    \label{cor312}
Let $\tilde f:\mathbb{R}^N\times \mathbb{R}\to \mathbb{R}$ be a function that satisfies assumptions
{\rm \ref{f1}} and {\rm \ref{f2}} with some function 
$\tilde \phi$ in place of $\phi$ and $p=1$ and such that $\tilde f(x,0) = 0$.
 Then,  there exist constants $C, \hat C>0$ such that   for all $u,v,w\in E_b$,
              \begin{align*}
              \Big|\int_{\mathds{R}^{N}}\hspace{-1mm} \big[\tilde f(x,u(x))- \tilde f(x,v(x))\big] w(x) \mathrm{d}x \Big|
               \leqslant C   \|u-v\|_{L_2} \|w\|_{ab} 
              \leqslant \hat C \|u-v\|_{ab}  \|w\|_{ab}.
\end{align*}
               In particular,
            \begin{align*}
              \Big|\int_{\mathds{R}^{N}}\hspace{-1mm} \tilde f(x,u(x)) w(x) \mathrm{d}x \Big|
              \leqslant \hat C \|u\|_{ab}  \|w\|_{ab}.   
\end{align*}
\end{corollary}
    \begin{proof}
    One has to apply H\"older's inequality in \eqref{ES_RHS-1111}, setting $p=1$ and $\phi = \tilde \phi$,
     with $q_1=\alpha$, $q_2 = 2(1 - \frac2\alpha)^{-1}$,  and $q_3 = 2$. We obtain the following bound for the
     right-hand side of \eqref{ES_RHS-1111}:
\begin{align}
\label{est-4444b}
      \|w\|_{\tilde \phi}^\frac2\alpha
      \|w\|_{L_2}^{1-\frac2\alpha} \|u-v\|_{L_2}
      \leqslant C \|w\|_{ab}  \|u-v\|_{L_2},
\end{align}
      for some constant $C>0$.
    \end{proof}
Now we are ready to define the energy functional associated
with equation \eqref{PS} which is equivalent
to the original system \eqref{P}. Namely, we define
    \begin{equation}
    \label{energy}
      J: E_b\to \mathbb{R},\quad   J(u):=\frac{1}{2}\|u\|_{ab}^{2}- \Psi(u).
    \end{equation}
    Expressing the objects in \eqref{energy} through their definitions
    gives
    \begin{equation*}
        J(u)= \frac{1}{2}\left(\int_{\mathds{R}^{N}} |\nabla u|^{2} +a(x)u S_{b}u \right) \mathrm{d}x -
        \int_{\mathds{R}^{N}}  F(x,u(x)) \,\mathrm{d}x,
    \end{equation*}
    where $F(x,u)$ is defined in \eqref{psi1111}.
    The immediate consequence of Lemma \ref{frch-2222} is the following
    result.
    \begin{lemma}
    The Fr\'echet derivative of the energy functional $J: E_b\to \mathbb{R}$,
    along the space $(E_b, \|\cdot\|_{ab})$, is
    \begin{align}
    \label{DJ}
    DJ(u) h  = \<u,h\rangle_{ab} + \int_{\mathds{R}^{N}} f(x,u(x)) h(x) \mathrm{d}x, \quad u,h\in E_b.
\end{align}
    \end{lemma}
    \begin{lemma}
    \label{weak-solutions-1111}
    Let  $u\in E_b$ be a critical point of the functional $J$, given by \eqref{energy}. Then, $u$
    is a  weak solutions of the non-local equation \eqref{PS} and $(u,S_b u)$ is a weak solution
    of system \eqref{P}.
    \end{lemma}
    \begin{proof}
    The proof immediately follows from the fact that the critical point $u$ 
    satisfies the equation $DJ(u) h=0$ for all $h\in E_b$. By \eqref{DJ}, this is the same as
    equation \eqref{weak-solution-2222}.
    \end{proof}
\section{Existence of solution for sign-changing coefficients: Proof of Theorem \ref{teo1} }
\label{theorem17}
\begin{lemma}
\label{lem51-111}
  Let the function $f$ satisfy {\rm \ref{f1}--\ref{f4}} and 
  let $\Psi$ be defined by \eqref{psi1111}. Further
        let $b\in \Upsilon_s(\mathds{R}^N)$ for some $s\in (2,2^*)$ and $a(x)$ satisfy {\rm \ref{f0a}}. 
        Then, 
\begin{equation*}
        D\Psi(u) u \geqslant \mu_0 \Psi(u)
        \quad \forall \; u \in E_b,
    \end{equation*}
    where $D\Psi$ is the Fr\'echet derivative of $\Psi$ and $\mu_0$ is the constant from hypothesis {\rm \ref{f4}}.
    \end{lemma}
    \begin{proof}
    By Lemma \ref{frch-2222} and hypothesis \ref{f4},
    \begin{align*}
     D\Psi(u) u = \int_{\mathds{R}^N}u f(x,u)\, \mathrm{d}x \geqslant \mu_{0}\int_{\mathds{R}^N} F(x,u) \,\mathrm{d}x
            = \mu_0\Psi(u). 
\end{align*}       \end{proof}
     \begin{lemma} \label {mountain pass geometry}
    Let $f$ be a function satisfying conditions {\rm \ref{f1}--\ref{f4}}. Further
        let $b\in \Upsilon_s(\mathds{R}^N)$ for some $s\in (2,2^*)$ and $a(x)$ satisfy {\rm \ref{f0a}}. 
     Then, the functional $J(u)$, given by \eqref{energy}, with $\Psi$ defined by \eqref{psi1111},
     possesses the properties:
     $ J(0)=0$ and there exist $\rho>0$, $g\in E_b$
    with $\|g\|_{ab}>\rho$ such that
    \begin{align*}
    \inf_{\{\|u\|_{ab}=\rho\}}  J(u)>0 \quad \text{and} \quad          J(g)<0.
    \end{align*}
    \end{lemma}
      \begin{proof}[Proof of Lemma \ref{mountain pass geometry}]
       Clearly, $J(0) = 0$.   
   By  Corollary \ref{cor-3.11}, there exist a constant $C>0$ such that
        \begin{align*}
                 J(u)&=\frac{1}{2}\|u\|_{ab}^2- \int_{\mathds{R}^{N}}F(x,u) \, \mathrm{d}x\geqslant 
                  \frac{1}{2}\|u\|_{ab}^2-\frac{1}{\mu_{0}}\int_{\mathds{R}^{N}}f(x,u)u \, \mathrm{d}x\\
                 &\geqslant  \frac{1}{2}\|u\|_{ab}^2-\frac{C}{\mu_{0}}\|u\|_{ab}^{p+1} 
                 \geqslant \Big(\frac{1}{2}-\frac{C}{\mu_0}\|u\|_{ab}^{p-1}\Big)\|u\|_{ab}^2.
\end{align*}
Therefore,
    \begin{align*}
    \inf_{\{\|u\|_{ab}=\rho\}}  J(u) \geqslant \Big(\frac{1}{2}-\frac{C}{\mu_0}\,\rho^{\,p-1}\Big)\rho^2 > 0 \quad \text{for} \;\;
    \rho\in \Big(0,\Big(\frac{\mu_0}{2C}\Big)^\frac1{p-1}\Big).
\end{align*}
 By \ref{f4},  for any $\kappa>0$ and $u\geqslant \kappa$,
              \begin{align*}
              \int_\kappa^u \frac{\partial_u F(x,w)}{F(x,w)} dw \geqslant \mu_0\int_\kappa^u \frac{dw}w\,.
\end{align*}
     This implies 
     \begin{align*}
     F(x,u) \geqslant \frac{F(x,\kappa) u^{\mu_0}}{\kappa^{\mu_0}} \quad \text{for all} \;\; u \geqslant \kappa.
\end{align*}    
     Likewise, if $u\leqslant -\kappa$, we obtain
     \begin{align*}
     F(x,u) \geqslant \frac{F(x,-\kappa) |u|^{\mu_0}}{\kappa^{\mu_0}}, \quad \text{for all} \;\; u \leqslant -\kappa.
\end{align*}  
     Let $u_{0}\in {{\rm C}}_0^\infty(\mathbb{R}^N)$ be a function with compact support, 
     $u_0\not\equiv 0$, 
     and let $\kappa>0$ be such that 
              $|D_\kappa|>0$, where $D_\kappa:=\{x\in {\rm supp}\, u_0: |u_0(x)| \geqslant \kappa\}$.
              Note that
              ${\rm supp}\, F(\cdot,u_0(\cdot)) \subset {\rm supp}\, u_0$.
              Let $\lambda\geqslant1$.    Then, on $D_\kappa$,
      \begin{align*}
      F(x,\lambda u_{0}) \geqslant \min\big\{F(x,\kappa), F(x,-\kappa)\big\} \lambda^{\mu_0} 
      \frac{|u_0|^{\mu_0}}{\kappa^{\mu_0}} \geqslant \min\big\{F(x,\kappa), F(x,-\kappa)\big\} \lambda^{\mu_0}. 
\end{align*}
      By \ref{f4} and Remark \ref{rm15},
      \begin{align*}
      0< \int_{D_\kappa}\min\big\{F(x,\kappa), F(x,-\kappa)\big\}  dx \leqslant
      \frac{C_0}{\mu_0} (\kappa^2+\kappa^{p-1}) 
      \int_{D_\kappa} (1+ \phi^\frac1\alpha) dx.
\end{align*}
Therefore,       
        \begin{equation}\begin{aligned}
        \label{ambro}
              J(\lambda u_{0}) = & \, \frac{\lambda^2}{2}\|u_0\|_{ab}^2 - 
              \int_{{\rm supp}\, u_0}F(x,\lambda u_{0}) \, \mathrm{d}x
              \leqslant \frac{\lambda^2}{2}\|u_0\|_{ab}^2 - \int_{D_\kappa}F(x,\lambda u_{0}) \, \mathrm{d}x \\
              \leqslant &\, \frac{\lambda^2}{2}\| u_0 \|_{ab}^2
              -  \lambda^{\mu_0}   \int_{D_\kappa}  \min\big\{F(x,\kappa), F(x,-\kappa)\big\} d x 
              \rightarrow-\infty \quad \text{as } \lambda\rightarrow + \infty,\end{aligned}
\end{equation}
              where the convergence holds since $\mu_0>2$. This completes the proof of the lemma.
    \end{proof}
    \begin{remark}
    \label{mpg-1111}
    \rm
We say that a functional $\mathcal J\in {{\rm C}}(X,\mathbb{R})$, given on a Banach space 
$(X, \|\cdot\|)$, has a {\it mountain pass geometry} if there exist $g\in X$ and $\rho>0$ with $\|g\|\geqslant \rho$ such that
\begin{align*}
   \inf_{\{\|u\| =\rho\}} \mathcal J(u) > \mathcal J(0) \geqslant \mathcal J(g).
\end{align*}
Thus, Lemma \ref{mountain pass geometry} tells us that the functional $J$, given by \eqref{energy} 
with $\Psi$ defined by \eqref{psi1111}, has a mountain pass geometry.
Therefore, we are in a position to apply the Mountain pass theorem
\cite[Theorem 2.10]{zbMATH00921895}
(see also its original version \cite[Theorem 2.1]{MR0370183}).
     \end{remark}
    \begin{lemma}
     \label{FDJ-1111}
     For the Fr\'echet derivative of $J$ one has
\begin{align*}
DJ(u)  = u - D\Psi(u).
\end{align*}
     \end{lemma}
     \begin{lemma}
     \label{lem-Eb}
        Let the function $f$ satisfy conditions {\rm \ref{f1}} and {\rm \ref{f2}},
        $b\in \Upsilon_s(\mathds{R}^n)$ for some $s\in (2,2^*)$, and $a(x)$  satisfy {\rm \ref{f0a}}. 
        Then, the map $E_b \to E_b$, $u\mapsto D\Psi(u)$ is compact.
 \end{lemma}
 \begin{proof}
 From Lemmas \ref{frch-2222} and \ref{A12} it follows that there exists a constant $C>0$ 
 and $s\in (2,2^*)$ such that
 \begin{align}
 \label{ineq-comp}
 \|D\Psi(u) - D\Psi(v)\|_{ab} \leqslant C  \|u\|^{p-1}_{ab} \|u-v\|_{L_s}. 
\end{align}
 It suffices to show that if $u_{n} \rightharpoonup u$ weakly in $(E_{b}, \|\cdot \|_{ab})$, then $D\Psi(u_{n}) \rightarrow D\Psi(u)$ 
 strongly in $(E_b, \|\cdot \|_{ab})$.  By Proposition \ref{compa}, $u_n\to u$ strongly
        in $L_s(\mathbb{R}^N)$ for any $s\in (2,2^*)$. This and \eqref{ineq-comp} imply the statement of the lemma.
 \end{proof}
    For the proof of Theorem \ref{teo1} we need the following two results. The first one
    can be found in \cite[Theorem 2.10]{zbMATH00921895} with its original version
    obtained in \cite{MR0370183}.
    \begin{theorem}[Mountain pass theorem of Ambrosetti-Rabinowitz (1973)]
    \label{mpt-AR}
    Let $\mathcal J\in {{\rm C}}^1(X,\mathbb{R})$ be a functional on a Banach space $X$ which has a
    mountain path geometry (see Remark \ref{mpg-1111}). Assume that
    $\mathcal J$ satisfies the  Palais-Smale condition $(PS)$, that is, any sequence $\{u_n\}\subset X$
        such that $\mathcal J(u_n)\to \mathfrak c\in\mathbb{R}$ and $D\mathcal J(u_n)\to 0$ contains a convergent subsequence.
        If
        \begin{align}
        \label{nc-2222}
       \mathfrak  c: = \inf_{\gamma \in \Gamma} \max_{t \in [0, 1]} \mathcal J(\gamma(t)),
\end{align}
    where
    \begin{align}
    \label{gm-1111}
    \Gamma: = \left\{\gamma \in {{\rm C}}([0, 1], X) : \gamma(0) = 0, \, \gamma(1) = g \right\},
\end{align}
    then $\mathfrak c$ is a critical value of $\mathcal J$.
    \end{theorem}
    \begin{remark}
    \rm
    In the above theorem $g\in X$ is the element defined in Remark \ref{mpg-1111}.
    \end{remark}
    The next result is due to \cite[Theorem 5.1]{zbMATH00044449}. We need this result to guarantee the existence
    of a Palais-Smale sequence.
    \begin{theorem}
    \label{PS-seq-1111}
    Let $\mathcal J\in {{\rm C}}^1(X,\mathbb{R})$ be a functional on a Banach space $X$,
    $T$ be a compact metric space, and $T_0\subset T$ be its closed subspace.
    Further let $\gamma_0\in {{\rm C}}(T_0,X)$ and
    \begin{align*}
    \tilde \Gamma:=\big\{\gamma\in {{\rm C}}(T,X): \;\; \left.\gamma\right|_{T_0} = \gamma_0\big\}.
\end{align*}
    Assume that
    \begin{align}
    \label{condition-max-1111}
    \max_{t\in T} \mathcal J(\gamma(t)) > \max_{t\in T_0} \mathcal J(\gamma(t)) \quad
    \forall \; \gamma\in \tilde \Gamma.
\end{align}
    Then, there exists a sequence $\{u_n\}\subset X$ such that as $n\to\infty$,
    \begin{align*}
    \mathcal J(u_n) \to \inf_{\gamma\in\tilde \Gamma} \max_{t\in T} \mathcal J(\gamma(t))
    \quad \text{and}  \quad
    DJ(u_n) \to 0.
\end{align*}
    \end{theorem}
   
 \begin{proof}[Proof of Theorem \ref{teo1}]
For the proof, we will apply Theorem \ref{mpt-AR} with respect to the functional
 $J$, given by \eqref{energy} with $\Psi$ defined by \eqref{psi1111},
 the number $\mathfrak c\in\mathbb{R}$ defined with respect to the space $X=E_b$,
 and the element $g\in E_b$ constructed in Lemma \ref{mountain pass geometry}.
 From Lemma \ref{mountain pass geometry} and Remark \ref{mpg-1111}
 we know that $J$ has the mountain pass geometry. Thus, to apply
 Theorem \ref{mpt-AR}, it remains to verify
 the Palais--Smale condition $(PS)$ for the functional $J$
 with respect to the number $\mathfrak c$ given by \eqref{nc-2222}.
      Let $\{u_n\}$ be a sequence satisfying $(PS)$.
        We claim that $\{\|u_n\|_{ab}\}$ is a bounded sequence.
        If this is the case, then, by Lemma \ref{FDJ-1111},
        $\{u_n\}$ contains a convergent subsequence. Indeed,
        by Lemma \ref{lem-Eb}, $E_b\to E_b$, $u\mapsto D\Psi(u)$ is a compact map. Therefore,
        there exists a convergent subsequence $\{D\Psi(u_{n_k})\}$ of
        $\{D\Psi(u_n)\}$. Hence,
        by Lemma \ref{FDJ-1111},  $u_{n_k} = DJ(u_{n_k})- D\Psi(u_{n_k})$
        is convergent. Let $u_0\in E_b$ be the limit of $u_{n_k}$. Then, condition $(PS)$
        implies that $u_0$ is a critical point of $J$ at level $\mathfrak c$.

       Let us prove that $\{\|u_n\|_{ab}\}$ is bounded. 
        By Lemma \ref{lem51-111},        $D\Psi (u_n) u_n \geqslant \mu_0 \Psi(u_n)$.
        Therefore,
        \begin{equation}\begin{aligned}
        \label{47-1111}
        \mu_0 J(u_n)-  DJ(u_n) u_n  = &\,
        \frac{\mu_0}2 \|u_n\|^2_{ab}  - \mu_0 \Psi(u_n) - \|u_n\|^2_{ab}
        + D\Psi (u_n) u_n \\
        \geqslant &\,  \Big(\frac{\mu_0}2 - 1\Big)\|u_n\|^2_{ab}.\end{aligned}
\end{equation}
On the other hand,
since $\{u_n\}$ satisfies $(PS)$,
for sufficiently large $n$ we have
\begin{equation}\begin{aligned}
\label{48-1111}
\mu_0 J(u_n)-  DJ(u_n) u_n  \leqslant   &\,  \mu_0 J(u_n)
+ \| DJ(u_n)\|_{ab}\|u_n\|_{ab} \leqslant 2\mu_0 \mathfrak c + \Big(\frac{\mu_0}4
- \frac12\Big) \|u_n\|_{ab}.\end{aligned}
\end{equation}
Represent $\{u_n\} = \{u'_n\}\cup \{u''_n\}$, where
       $\|u'_n\|_{ab}>1$ and $\|u''_n\|_{ab}\leqslant 1$. 
        Comparing \eqref{47-1111} and \eqref{48-1111} for the sequence $\{u'_n\}$, we conclude that
        $\{\|u'_n\|_{ab}\}$ is bounded. Hence, $\{\|u_n\|_{ab}\}$ is bounded.

        It remains to prove the existence of a sequence $\{u_n\}$
         satisfying condition $(PS)$ (a Palais-Smale sequence) with
         respect to the number $\mathfrak c$ defined by \eqref{nc-2222}.
          We apply Theorem \ref{PS-seq-1111} with $T=[0,1]$,
        $T_0=\{0,1\}$, $X=(E_b, \|\cdot\|_{ab})$, and
        $\gamma_0: \{0,1\} \to E_b$, $\gamma_0(0) = 0$ and $\gamma_0(1) = g$,
        where $g\in E_b$ is the element constructed in Lemma \ref{mountain pass geometry}.
        Note that, with the above definitions, $\tilde \Gamma = \Gamma$, and hence,
        \begin{align*}
        \inf_{\gamma\in\tilde \Gamma} \max_{t\in T} \mathcal J(\gamma(t)) = \mathfrak c,
\end{align*}
        where $\mathfrak c$ is defined by \eqref{nc-2222}.
        We also note that any curve $\gamma\in \Gamma$
    intersects the sphere $S_\rho:=\{\|u\|_{ab} = \rho\}$.
    Indeed, the point $\gamma(0)$ coincides with the origin and therefore
    lies inside $S_\rho$; the point $\gamma(1)=g$ lies outside of $S_\rho$ by Lemma
    \ref{mountain pass geometry}. Lemma \ref{mountain pass geometry} also implies
    that for any $\gamma\in \tilde \Gamma = \Gamma$,
    \begin{align*}
    \max_{t\in [0,1]} J(\gamma(t))> \max_{t\in \{0,1\}} J(\gamma(t))
\end{align*}
    since
    \begin{align}
    \label{67-1111}
    \max_{t\in [0,1]} J(\gamma(t)) \geqslant \max_{t\in [0,1], \gamma(t)\in S_\rho} J(\gamma(t))
    \geqslant \inf_{u\in S_\rho} J(u) > 0
     \quad \text{and} \quad
    \max_{t\in \{0,1\}} J(\gamma(t)) = J(0) = 0.
\end{align}
    By Theorem \ref{PS-seq-1111}, there exists a Palais-Smale sequence
    $\{u_n\}$ with respect to the number $\mathfrak c$ given by \eqref{nc-2222}.
    We also note that $\mathfrak c\ne 0$. Indeed, taking infimum in \eqref{67-1111} over $\gamma\in \Gamma$, we obtain
    that
    \begin{align*}
      \mathfrak  c \geqslant \inf_{u\in S_\rho} J(u) > 0.
\end{align*}
      This implies that the critical point $u_0$, which is a solution to \eqref{PS}, is non-trivial, and hence,
      the solution $(u_0, S_b u_0)$ to system \eqref{P} is non-trivial.
       \end{proof}

\section{Componentwise positive and negative solutions: Proof of Theorem \ref{teo}}
\label{Proof of the main result}
\subsection{An abstract result}
In this subsection, we state the Weth theorem (Theorem \ref{Abstract Result}), 
an important auxiliary result that will be applied to the energy functional \eqref{energy}  
on an appropriately chosen Hilbert space.

Let $(H,\langle\cdot, \cdot\rangle)$ be a Hilbert space
and let $\Phi\in {{\rm C}}^1(H,\mathbb{R})$.
Define the functional
\begin{align}
\label{J-1111}
\mathcal J: H\to \mathbb{R}, \quad \mathcal J(u)=\frac{1}{2}\|u\|^2-\Phi(u),
\end{align}
where $\|\cdot\|:= \langle\cdot, \cdot\rangle^\frac12$.
Recall that a subset ${\mathcal K} \subset H$ is called a  cone
if $t {\mathcal K} \subset {\mathcal K}$ for all $t>0$. Furthermore,
\begin{align}
\label{dual-1111}
{\mathcal K}^o:=\{u \in H ;\langle u, v\rangle \leqslant 0, \;\; \forall v \in {\mathcal K}\},
\end{align}
denotes the polar cone. 
Let $\mathcal K$ be a closed cone. Recall that the projection on $\mathcal{K}$,
$P_{{\mathcal K}}: H \rightarrow \mathcal K$, $u\mapsto P_{\mathcal K} u$
is  uniquely determined by the identity
$$
\|u-P_{\mathcal K} u\|=\min _{v \in \mathcal K}\|u-v\| .
$$

The following result is due to Weth (see \cite[Theorem 2.6]{MR2262254} and 
\cite[Proposition 2.4]{MR2262254}).
\begin{theorem}\label {Abstract Result}
Let the functional $\mathcal J$ be defined by \eqref{J-1111} with
$\Phi: H \to \mathbb{R}$ being Fr\'echet differentiable, and let $\mathcal K$ be a closed, convex,
pointed cone $($the latter means that $\mathcal K \cap (-\mathcal K) = \{0\})$.
Assume that $\mathcal A: = D\Phi$, where $D$ is the Fr\'echet derivative,
 satisfies the conditions  {\rm \ref{A1}--\ref{A4}:}
\begin{enumerate}
\item[\namedlabel{A1}{\rm ($A_1$)}]  $\mathcal A: H\to H$ is compact, locally Lipschitz, and $\mathcal A(0)=0$.
    \item[\namedlabel{A2}{\rm ($A_2$)}]
    There exist constants $C^*\!, q^*>0, q_{*} \in(0,1), \ell >1$, and $\eta>2$ such that
    \begin{align}
\label{A_2_1}
        &\langle \mathcal A(u), u\rangle \geqslant \eta \Phi(u)-C^{*},
        \quad \forall \; u \in H \\
        \noalign{\vskip-4pt}
        & \hspace{-27mm} \text{and}\notag\\ 
        \noalign{\vskip-8pt}
     &|\langle \mathcal A(u), v\rangle| \leqslant
        \Big(q_{*}\|u\|+ q^*\|u\|^{\ell}\Big) \|v\| \quad \forall \; u, v \in H. \label {A_2_2}
\end{align}
\item[\namedlabel{A3}{\rm ($A_3$)}]  
$ \langle \mathcal A(u), v\rangle \leqslant\big\langle\mathcal A\big(P_{\mathcal K^o} u\big), v \big\rangle \quad \forall u \in H$ and $v \in \mathcal K^o$,
\item[\namedlabel{A4}{\rm ($A_4$)}]  
$\langle \mathcal A(u), w\rangle \leqslant\big\langle \mathcal A\big(P_{(-\mathcal K)^o} u\big), w\big\rangle \quad \forall u \in H$ and $w \in(-\mathcal K)^o$. 
\end{enumerate}
Further assume there exists a continuous path $h:[0,1] \rightarrow H$ such that
\begin{align}
\label{cp-pro-1111}
h(0) \in \mathcal K, \quad h(1) \in -\mathcal K \quad \text { and } \quad \mathcal J(h(t))<0, \quad \forall t \in[0,1] .
\end{align}
Then, $\mathcal J$ has at least three nontrivial critical points $u_1, u_2, u_3$, where $u_1 \in \mathcal K, u_2 \in - \mathcal K$ 
and $u_3 \in H \setminus (\mathcal K \cup -\mathcal K)$.
\end{theorem}

   \subsection{Scalar product $\langle\, \cdot \,,\, \cdot \,\rangle_{ab}^*$ in $E_b$ and the polar cone}
     In this subsection, we define the cone $K$ and its polar  cone $K^o$ which are essential 
for the application of Theorem~\ref{Abstract Result} in our framework.
     Introduce the cone
    \begin{equation}
    \label{cone1111}
        {K}: =
        \left\{  u\in E_b: \,  u\geqslant0 \, \text{ a.e. in  } \, \mathds{R}^{N}  \right\}.
    \end{equation}
    The polar cone $K^o$ depends on the choice of a scalar product in $E_b$. 
    In what follows, we define a new scalar product $\langle\, \cdot \,,\, \cdot \,\rangle_{ab}^*$  that
    will allow us to verify assumptions \ref{A3} and \ref{A4} in Theorem \ref{Abstract Result}.
   For this, we need Lemmas \ref{big-inverse}
  and \ref{scalar-1111}. Their proof is postponed to Subsection \ref{proofs}.
     \begin{lemma}[Strong maximum principle for the operator 
 $(-\Delta+ aS_b+ b-2\beta^\frac12 a)^{-1}$]
 \label{big-inverse} 
Let $b\in \mathcal T(\mathbb{R}^N)$. Assume that 
$e(x):= b(x)-2\beta^\frac12 a(x)$
 satisfies {\rm \ref{f0ab}}. 
 Further assume that $a>0$ a.e. 
 Then, 
 \begin{align*}
 \big(\!-\!\Delta+ a(x)S_b + e(x)\big)^{-1}
\end{align*}
 exists and is a continuous operator $E_b\to E_b$ that maps $ K$ to $ K$. 
 Furthermore, $\!-\Delta+ a(x)S_b + e(x)$ satisfies the strong maximum principle, that is,
 if $v \in E_b$,  then
    \begin{align*}
     -\Delta v+ aS_b(v)+e(x) v \geqslant 0 \quad \text{implies that either} \;\; v> 0
    \;\; \text{or} \;\; v = 0 \;\; \text{a.e. on } \; \mathbb{R}^N.
\end{align*}
 \end{lemma}
 \begin{lemma}
 \label{scalar-1111}
 Let $b\in \mathcal T(\mathbb{R}^N)$, $a(x)$ satisfy {\rm \ref{f0a}}, and $e(x)$ satisfy 
 {\rm \ref{f0ab}}. Then,
 \begin{align*}
 \<u,v\rangle_{ab}^*: = \int_{\mathbb{R}^N}
 \big(\nabla u \nabla v + aS_b u v + e uv\big) {\rm d}x
\end{align*}
 is a scalar product in $E_b$. Moreover, the corresponding norm 
 
 \begin{align*}
 \|u\|_{ab}^* = \Big(\|u\|^2_{ab} + \int_{\mathbb{R}^N} e(x) u^2 {\rm d}x \Big)^\frac12.
\end{align*}
 is equivalent to the norm $\|\cdot\|_{ab}$.
 \end{lemma}
 Introduce the polar cone
   \begin{equation}
    \label{ad-cone1111}
        {K}^{o} =\left\{  u\in E: \, \langle  u, v \rangle^*_{ab}\leqslant 0 \, \text{ for all } \, v\in {K}  \right\}.
    \end{equation}
     In view of Lemma \ref{big-inverse}, we transform the original nonlocal equation \eqref{PS}
     to an equivalent one:
     \begin{align}
     \label{new-eq}
     &-\Delta u + a(x)S_{b}(u) +  e(x) u = \tilde f(x,u),\\
 &\text{ where} \quad 
  \tilde f(x,u): = f(x,u) +e(x)u. \notag
\end{align}
 \subsection{Auxiliary lemmas and proof of Theorems \ref{t-main} and  \ref{teo}}
 First, we agree about the following notation.
     Let  $J: E_b\to \mathbb{R}$  be the functional defined by \eqref{energy}
     with the functional $\Psi: E_b\to \mathbb{R}$ given by \eqref{psi1111}.
    
     Lemmas \ref{lem-J-1111}--\ref{EDO} below, whose proof
    is postponed to Subsection \ref{proofs} for the clarity
    of this exposition, are steps in the proof of Theorem \ref{teo}.
     \begin{lemma}
     \label{lem-J-1111}
     Let $\tilde J(u) = \frac12{\|u\|^*_{ab}}^2-\tilde \Psi(u)$,
     where $\tilde \Psi(u): = \int_{\mathbb{R}^N} \tilde F(x,u(x)) dx$ and
     $\tilde F(x,u): = F(x,u) + \frac12 e(x) u^2$.
      Then,   for all $u\in E_b$,   $\tilde J(u) = J(u)$, 
that is, the energy functional associated to equations  \eqref{PS} and \eqref{new-eq} is the same.
     \end{lemma}
\begin{lemma}
    \label{frch-1111}
    The Fr\'echet derivative $D\tilde \Psi$  along
    the space $(E_b, \|\cdot\|^*_{ab})$ equals
    \begin{align*}
    D\tilde \Psi(u) h =  D\Psi(u) h + B(u,h), \quad u,h\in E_b
\end{align*}
    where 
\begin{align*}
B(u,h): = \int_{\mathds{R}^{N}} e(x)u(x) h(x) \mathrm{d}x.
\end{align*}
    Furthermore, the map 
    \begin{align*}
    A: E_b\to E_b, \quad A: = D \tilde \Psi,
\end{align*}
 is given by
    \begin{align}
\label{map-A-1111}
\big<A(u),h\big>^*_{ab} = 
\int_{\mathds{R}^{N}} f(x,u(x)) h(x) \mathrm{d}x   +  B(u,h),
\quad \text{for} \;\; u,h\in E_b.
\end{align}

\end{lemma}
\begin{lemma}[Weak Maximum Principle for the operator $-\Delta +b$]
\label{5.11-11}
Assume that  either $b\in \mathcal T(\mathbb{R}^N)$ or $b\geqslant 0$ $($a.e.$)$ is measurable, 
 with $b\ne 0$ a.e.
Let $v\in E_b$. Then
 \begin{align*}
    -\Delta v(x) + b(x) v(x) \geqslant 0 \quad \text{implies} \quad v \geqslant 0
 \;\; \text{a.e. on } \;\; \mathbb{R}^N.
\end{align*}
     \end{lemma}   
 \begin{remark}
 \label{rem-3333}
 \rm Let $v\in E_b$. We say that the inequality 
 $-\Delta v(x) + b(x) v(x) \geqslant 0$ is fulfilled, if for any 
 function $\varphi\in{{\rm C}}^\infty_0(\mathbb{R}^n)$, $\varphi\geqslant 0$, 
 \begin{align}
 \label{weak-1111}
 \int_{\mathbb{R}^N} [(\nabla v, \nabla \varphi) +  bv\varphi] dx \geqslant 0.
\end{align}
 By Lemma \ref{type12-1111}, \eqref{weak-1111} also holds for any
 $\varphi\in E_b$, $\varphi\geqslant 0$.
 \end{remark}    
    \begin{lemma}
    \label {A11}
     Let $b\in \Upsilon_2(\mathds{R}^N)$. Further let
        the function $f$ satisfy conditions {\rm \ref{f1}} and {\rm \ref{f2}}
        and $a(x)$  satisfy {\rm \ref{f0a}}. 
   Then, the map $A: E_b\to E_b$, given
    by \eqref{map-A-1111}, is locally Lipschitz and compact.
    \end{lemma}
    \begin{lemma}
    \label{LA33}
    Let $f$ satisfy {\rm \ref{f3}}. Then, $A(0) = 0$.
    \end{lemma}

    \begin{lemma}\label {LA2} 
    Let $f$ satisfy conditions {\rm \ref{f1}--\ref{f4}}. Further
        let $b\in \Upsilon_2(\mathds{R}^N)$ and $a(x)$ satisfy {\rm \ref{f0a}}. 
          Then, the map $A$ (in place of $\mathcal A$) verifies the conditions in {\rm \ref{A2}}.
    \end{lemma}

   \begin{lemma}[Strong Maximum Principle for the operator
$-\Delta +b$]
\label{5.22-22}
Let the condition of the weak maximum principle $($Lemma \ref{5.11-11}$)$
be fulfilled, and let $v\in E_b$. Then
    \begin{align*}
    -\Delta v(x) + b v(x) \geqslant 0 \quad \text{implies that either} \;\; v> 0
    \;\; \text{or} \;\; v = 0 \;\; \text{a.e. on } \;\; \mathbb{R}^N.
\end{align*}
     \end{lemma}     
     \begin{corollary}[Corollary of Lemma \ref{5.22-22}]
    \label{strong-mp}
    Let the condition of the weak maximum principle $($Lemma \ref{5.11-11}$)$
be fulfilled and
 let $a>0$ a.e., 
 $u\in  K$, $u\ne 0$ a.e. Then, 
 \begin{align*}
 \quad S_b u > 0  \;\; \text{a.e. on} \;\; \mathbb{R}^N. 
\end{align*}
 \end{corollary}
 \begin{lemma}[Moreau's decomposition theorem]
    \label{moreau}
     Let $H$ be a Hilbert space, and let
    $K\subset H$ be a closed, convex cone, and $K^o$ be its polar cone
    defined by \eqref{dual-1111}.
    Then,
    \begin{align*}
    v = P_{K} u \;\; \Leftrightarrow \;\;
    \big[v\in K, \; (u-v,v)_H = 0, \; \text{and} \;\;  u-v\in K^o\big].  
\end{align*}
    \end{lemma}
     \begin{lemma}
    \label{lem4.8}
Let  the polar cone ${K}^{o}$ be defined by \eqref{ad-cone1111}.
Then,
 \begin{align*}
    {K}^{o} \subset -{K} 
\end{align*}
        \end{lemma}
    \begin{lemma}\label {LA3}
        Let $f$ satisfy {\rm \ref{f1}--\ref{f5}}, $b\in \Upsilon_2(\mathds{R}^N)$,
        and $a(x)$ satisfy {\rm \ref{f0a}}. 
        Then, for all $u\in E_b$, $v\in {K}^{o}$, 
        and $w\in  (-K)^o$,
        \begin{equation*}
            \langle A(u),v\rangle \leqslant \langle A \left(  P_{K^o} u\right),v \rangle \quad \text{and} 
            \quad \langle A(u),w\rangle \leqslant \langle A \left(  P_{(-K)^o} u\right),w \rangle.
        \end{equation*}
         In particular, \ref{A3} and \ref{A4} are fulfilled for the map
         $A$ in place of $\mathcal A$.
    \end{lemma}

    \begin{lemma}\label {EDO}
       Let the function $f$  satisfy conditions {\rm \ref{f1}--\ref{f4}}. Further
        let $b\in \Upsilon_2(\mathds{R}^N)$ and $a(x)$ satisfy {\rm \ref{f0a}}. 
        Then, there exists a continuous path $h:[0,1] \rightarrow E_{b}$ 
        such that $h(0)(x) \geqslant 0$ a.e. 
        $J(h(t))<0$ for all $t \in[0,1]$.
    \end{lemma}
    \begin{proof}[Proof of Theorem \ref{teo}]
    Let us show that Theorem \ref{teo}
    follows from Theorem \ref{Abstract Result}. Note that
    assumptions \ref{A1}-\ref{A4} are satisfied by the map $A$,
    in place of $\mathcal A$, with respect to the cone $K$ defined by \eqref{cone1111}
    and the Hilbert space $(H, \langle\cdot,\cdot\rangle^*_{ab})$.
  We also note that Lemmas \ref{A11} and \ref{LA33} imply \ref{A1},
   Lemma \ref{LA2} implies \ref{A2},  Lemma \ref{LA3} implies
   \ref{A3} and \ref{A4}. Finally, the existence of a continuous path
   $h:[0,1] \rightarrow E_{b}$ with the properties \eqref{cp-pro-1111}
   follows from Lemma \ref{EDO}.
     By Theorem \ref{Abstract Result}, the functional $J: E_b\to \mathbb{R}$, 
     given by  \eqref{energy} with the functional $\Psi$ given by \eqref{psi1111},
     possesses three nontrivial critical points $u_1, u_2, u_3$ such that
     $u_1\geqslant 0$ a.e., $u_2\leqslant 0$ a.e.,
     and $u_3$ is sign-changing. 
     Then  $(u_k, v_k)$, $v_k = S_b u_k$, $k=1,2,3$,
     are weak solutions to system  \eqref{P}. 
           Moreover, Corollary \ref{strong-mp} implies that $v_1 = S_b u_1 > 0$ and $v_2 = S_b u_2 < 0$. 
           Finally, since $u_1\geqslant 0$ and $u_2\leqslant 0$, by \ref{f5}, $\tilde f(x,u_1)\geqslant 0$ and $\tilde f(x,u_2)\leqslant 0$.
           Therefore,
     \begin{equation*}
         -\Delta u_1 + a(x)S_{b}(u_1) +  e(x) u_1 \geqslant 0  \qquad \text{and} \quad -\Delta u_2 + a(x)S_{b}(u_2) +  e(x) u_2 \leqslant 0. 
     \end{equation*}
     By Lemma \ref{big-inverse}, $u_1 > 0$ and $u_2 < 0$. 
     \end{proof}
        \begin{proof}[Proof of Corollary \ref{cor-teo1-and-teo}]
            Note that to prove the corollary it suffices to show that \ref{f22} implies \ref{f2}. By \ref{f22}, 
            for any function $\phi \colon \mathds{R}^N \rightarrow [0, +\infty)$ and all $\alpha > 0$, 
             \begin{equation*}
                |f(x, u)-f(x,v)|  \leqslant C_0\big(1+ \phi(x)^{1/\alpha}\big)\left(1+ \left(|u|^{p-1} +|v|^{p-1}\right)\right)|u-v|.
            \end{equation*}
            It remains to show that for any $p\in (1,2^*-1)$ there exists $\alpha>\max\{2,\frac{N}2\}$ such that
            $p\in \mathcal P_{\alpha,N}$, where the latter is defined by \eqref{PalfN}.
            Let $p \in \big(1, \frac{N+2-4/N}{N-2}\big) = \mathcal P_{N,N}$. Then \ref{f2} holds with $\alpha=N$.
            Now let $p \in \left[\frac{N+2-4/N}{N-2}, \frac{N+2}{N-2}\right)$. Define            
            \begin{equation*}
                \alpha_{(N,p)}: = \frac{4}{-(N-2)p + N + 2 - 2/N} \geqslant 2N.
            \end{equation*}
            Substituting the above expression for $\alpha_{(N,p)}$ to  \eqref{PalfN}, we obtain
            \begin{equation*}
                \mathcal{P}_{\alpha_{(N,p)}, N} = \left(1, 1 + \frac{2}{N}\right] \cup \left[p - \frac{2N-2}{N(N-2)}, p + \frac{2}{N(N-2)} \right),
            \end{equation*}
            which guarantees that $p \in \mathcal{P}_{\alpha_{(N,p)}, N}$. 
        \end{proof}

      \begin{proof}[Proof of Theorem \ref{t-main}]
            Theorem~\ref{t-main} follows directly from Corollary~\ref{cor-teo1-and-teo}.

            It is easy to see that if the function $f(u) = |u|^{p-1}u$, satisfies \ref{f1}, \ref{f3}, \ref{f4}, and \ref{f5}. 
            It remains to show that $f$ satisfies \ref{f22}.  Note that
            \begin{equation*}
                f(u) - f(v) = \int_0^1   f'(v + t(u - v)) \, \mathrm{d}t \cdot (u - v).
            \end{equation*}
            Therefore, $|f(u) - f(v)| \leqslant p \left( |u|^{p-1} + |v|^{p-1} \right)|u - v|$.
     \end{proof}

\subsection{Proofs of auxiliary lemmas}
\label{proofs}
We start by proving Lemma \ref{5.11-11} whose proof is required for Lemma \ref{big-inverse}.
\begin{proof}[Proof of Lemma \ref{5.11-11}]
 Note that if $v\in E_b$, then $v_-\in E_b$.  Indeed, it suffices to note
 that $\nabla v_- = \nabla v \mathds{1}_{\{v\leqslant 0\}}$. Setting $\varphi = v_-$ in \eqref{weak-1111}
 (recall that \eqref{weak-1111} holds for $\varphi\in E_b$ by Remark \ref{rem-3333}),
 we obtain 
 \begin{align}
 \label{neg-1111}
   \int_{\mathds{R}^N}\big[ (\nabla v, \nabla v_-) + b(x) v v_- \big] \, dx \leqslant 0.
\end{align}
 Consider the case $b\in \mathcal T(\mathbb{R}^N)$. By Lemma \ref{Embedded}, there is a constant $C>0$ such that
 \begin{align*}
 C\int_{\mathds{R}^N} v_-^2 \, dx \leqslant 
 \int_{\mathds{R}^N} \big[ |\nabla v_-|^2 + b(x) v_-^2 \big] dx \leqslant 0.
\end{align*}
 This implies that $v_-=0$ a.e. 
 
 Now consider the case when $b\geqslant 0$ (a.e.) is measurable with $b\ne 0$ a.e.  Inequality \eqref{neg-1111} implies
 $\nabla v_- = 0$ a.e., that is $v_-=const$ a.e. Since $\int_{\mathbb{R}^N}\!b\, v_-^2= 0$, we infer that $v_- = 0$ a.e.
\end{proof}
\begin{proof}[Proof of Lemma \ref{big-inverse}]
 Define $\mathcal B: = a S_b = \beta a(-\Delta + b)^{-1} a$,
so that
\begin{align*}
\mathcal B^{-1} = a^{-1} (-\Delta + b) (\beta a)^{-1} = (-a^{-1}\Delta + a^{-1}b) (\beta a)^{-1}.
\end{align*}
Let  $\eta: = -\beta^\frac12 a$ and $\theta:= b-\beta^\frac12 a$.
 We obtain
 \begin{align*}
 -\Delta+ aS_b + \theta + \eta = &\,  (-\Delta+\theta) +\eta \mathcal B^{-1}\mathcal B +  \mathcal B =
  (-\Delta+\theta) + (1+\eta\mathcal B^{-1})\mathcal B \\
 =&\, (-\Delta+\theta) + (\beta a - \eta a^{-1}\Delta +\eta a^{-1} b)(-\Delta+b)^{-1} a\\
 = &\,  (-\Delta+\theta) + \eta a^{-1}(-\Delta+\theta)(-\Delta+b)^{-1} a\\
 = &\,  (-\Delta+\theta)(1- \beta^\frac12 (-\Delta + b)^{-1} a)\\
 =&\,  (-\Delta+\theta)\beta^\frac12(-\Delta + b)^{-1} ( (-\Delta+b)\beta^{-\frac12}-a)\\
= &\,  (-\Delta+\theta)\beta^\frac12(-\Delta + b)^{-1} ( -\Delta+b+\eta)\beta^{-\frac12}\\
= &\,  (-\Delta+\theta)(-\Delta + b)^{-1} ( -\Delta+b+\eta).
\end{align*}
Therefore,
 \begin{align*}
  (-\Delta+ aS_b + \theta + \eta)^{-1} = &\, (-\Delta+b+\eta)^{-1}(-\Delta + b) (-\Delta+\theta)^{-1}\\ 
  = &\, (1- (-\Delta+b+\eta)^{-1}\eta) (-\Delta+\theta)^{-1}\\
   =   &\, \Big(1 + \beta^\frac12\big(-\Delta+ b-\beta^\frac12 a\big)^{-1}a\Big) 
   \big(-\Delta+ b-\beta^\frac12 a\big)^{-1}.
\end{align*}
  Let $d:=b-\beta^\frac12 a$. 
  The existence of the continuous inverse operators  $\big[(-\Delta+ d)^{-1}a\big]$ and $(-\Delta+ d)^{-1}$:
  $E_b\to E_b$ follows from Corollary \ref{lap-inv-1111}.
  Indeed, since $d > e\geqslant 0$,  to apply this corollary it suffices to prove the equivalence of the norms
  $\|\cdot\|_b$ and $\|\cdot\|_d$. It is straightforward to see that
  $\|\cdot\|_d \leqslant \|\cdot\|_b$.  
   On the other hand, since $d=e+\beta^\frac12 a$, by the equivalence 
   of the norms $\|\cdot\|_a$ and $\|\cdot\|_b$,  there exists
   a constant $C>0$ such that for all $u\in E_b$,
   \begin{align*}
\|u\|^2_d = \|u\|^2_e + \beta^\frac12 \int a u^2 \geqslant \min\{1,\beta^\frac12\} \|u\|^2_a + \int eu^2 
\geqslant C \|u\|^2_b.
\end{align*}
   Let $u\in K$. By Lemma \ref{5.22-22}, $(-\Delta+ d)^{-1}u>0$ and 
   by Corollary \ref{strong-mp}, $(-\Delta+ d)^{-1}au >0$.
 \end{proof}

 \begin{proof}[Proof of Lemma \ref{scalar-1111}]
 By Proposition \ref{Ne} and since $e \geqslant 0$ a.e., $\<u,v\rangle_{ab}^*$ is a scalar product 
 in $E_b$. In particular, by \ref{f0ab}, $\int_{\mathbb{R}^N} e uv dx$ is well defined for
 all $u,v\in E_b$.

Let us prove the equivalence of the norms $\| u \|_{ab}$ and  $\| u \|_{ab}^*$.
Since $e \geqslant 0$, the inequality $\| u \|_{ab} \leqslant \| u \|_{ab}^*$ is straightforward. 
On the other hand, since $0 \leqslant e \leqslant b$,
 there exists a constant $C_*>0$ such that
 \begin{align}
 \label{reverse}
\left(\| u \|_{ab}^*\right)^2 \leqslant \|u\|^2_{ab} + \int_{\mathbb{R}^N} b(x) u^2 {\rm d}x 
\leqslant \|u\|^2_{ab} + \|u\|_b^2 \leqslant C_* \|u\|^2_{ab}.
\end{align}
\end{proof}

\begin{proof}[Proof of Lemma \ref{lem-J-1111}]
The proof is straightforward.
\end{proof} 
\begin{proof}[Proof of Lemma \ref{frch-1111}]
It suffices to prove the Fr\'echet differentiability, 
with respect to the norm $\|\cdot\|_{ab}$, of the map
\begin{align}
\label{e2map}
E_b\to \mathbb{R}, \quad  u\mapsto \frac12 \int_{\mathbb{R}^N} e(x) u^2(x) dx
\end{align}
and prove
that its Fr\'echet derivative at point $u\in E_b$ is the functional $E_b\to \mathbb{R}$, $h\mapsto B(u,h)$.
By Corollary \ref{cor312}, there exist constants $C,\hat C>0$ such that
  \begin{align}
  \label{lip-com-1111}
              \Big|\int_{\mathds{R}^{N}}\hspace{-1mm} (e(x) u(x) - e(x) v(x)) h(x) \mathrm{d}x \Big|
               \leqslant C   \|u-v\|_{L_2} \|h\|_{ab} 
              \leqslant \hat C \|u-v\|_{ab}  \|h\|_{ab}.
\end{align}
Now the Fr\'echet differentiability of the map \eqref{e2map} follows by the
 same argument as in Lemma \ref{frch-2222}.
              
We also remark that by equivalence of the norms $\|\cdot\|_{ab}$ and $\|\cdot\|_{ab}^*$,
the map $\tilde \Psi$ is also Fr\'echet-differentiable with respect to the norm $\|\cdot\|_{ab}^*$.
\end{proof}
    \begin{proof}[Proof of Lemma \ref{A11}]
    It suffices to prove that the map $A$, given by \eqref{map-A-1111}, is Lipschitz and compact in any of the equivalent
    norms of $E_b$. By Lemma \ref{A12}, the map $E_b\to E_b$, $u\mapsto D\Psi(u)$ is Lipschitz and
    by Lemma \ref{lem-Eb}, this map  is compact.
    It suffices to prove that the Fr\'echet derivative of the map \eqref{e2map} is a locally Lipschitz and compact map.
    This follows from \eqref{lip-com-1111} by the same argument as in Lemma \ref{lem-Eb}.
        The lemma is proved.
          \end{proof}
\begin{proof}[Proof of Lemma \ref{LA33}]
The proof is straightforward.
\end{proof}
    \begin{proof}[Proof of Lemma \ref{LA2}]
    Condition \eqref{A_2_1}, with $C^* = 0$, follows immediately from Lemma  \ref{lem51-111}, the explicit 
    form of the map \eqref{e2map}, and its Fr\'echet derivative.
    
    Let us obtain \eqref{A_2_2}. Define the seminorm 
    $\|u\|^\sim_e: = \big(\int_{\mathbb{R}^N} e u^2 dx\big)^\frac12$.
    By Corollary \ref{cor-3.11}  and Lemma \ref{scalar-1111},
     there exists a constant $C>0$ such that for all $u,v \in E_b$,
\begin{align*}
\big|\big<A(u),v\big>^*_{ab}\big| \leqslant &\,
\Big|\int_{\mathds{R}^{N}} f(x,u(x)) v(x) \mathrm{d}x\Big| 
+ \Big|\int_{\mathds{R}^{N}} e(x)u(x) v(x) \mathrm{d}x\Big|\\
\leqslant &  C \|u\|^{p}_{ab} \|v\|_{ab} +   \|u\|^\sim_{e} \|v\|^\sim_{e}                           
  \leqslant\,\big( C{\|u\|^{*}_{ab}}^p + \|u\|^\sim_{e} \big) \|v\|_{ab}^{*}.
\end{align*}
Let us prove that there exists $q_* \in (0,1)$ such that
 $\|u\|^\sim_{e} \leqslant q_* \|u\|_{ab}^{*}$ for all $u \in E_b$. By \eqref{reverse}, 
\begin{equation*}
0 \leqslant {\|u\|^\sim_e}^2= {\|u\|_{ab}^{*}}^2 - \|u\|_{ab}^2 \leqslant \left(1 - C_*^{-1}\right)
{\|u\|_{ab}^{*}}^2
\end{equation*}
for all $u \in E_b$, which proves that $C_*\geqslant 1$. If $C_* > 1$, define
$q_* = 1-C_*^{-1} \in (0,1)$. 
If $C_* = 1$, then $\|\cdot\|^\sim_{e} \equiv 0$ and 
$\|u\|^\sim_{e}\leqslant q_* \|u\|_{ab}^{*}$ for all $u \in E_b$ and $q_*\in (0,1)$.
    \end{proof}

\begin{proof}[Proof of Lemma \ref{5.22-22}]
By Lemma \ref{5.11-11}, $u\geqslant 0$. Suppose $u=0$ a.e. on a set
$D\subset\mathbb{R}^N$ with $|D|>0$. Then,
there exists a ball $B_R(0)$ such that $|B_R(0)\cap D| > 0$.
This implies that $\essinf u|_{B_R(0)} = 0$. Furthermore, we note
for any $\varphi\in {{\rm C}}^\infty_0$, $\varphi\geqslant 0$,
\begin{align*}
 \int_{\mathds{R}^N}\big[ (\nabla u, \nabla \varphi) + b_+(x) u \varphi \big] \, dx
 \geqslant  \int_{\mathds{R}^N}\big[ (\nabla u, \nabla \varphi) + b(x) u \varphi \big] \, dx
  \geqslant 0.
\end{align*}
The operator $-\Delta + b_+$ satisfies the assumption of the
strong maximum principle in \cite[Theorem 8.19]{zbMATH01554166}
on the domain $\Omega = B_{2R}(0)$. Indeed, $b_+$ is essentially bounded on $\Omega$.
Also, since $u\geqslant 0$ a.e. on $\Omega$ and $u=0$ a.e. on $D$,
$\essinf u|_{B_R(0)} = \essinf u|_{\Omega} = 0$. By 
\cite[Theorem 8.19]{zbMATH01554166}, $u=0$ a.e.  on $\Omega$.
Consequently, $u=0$ a.e. on $\mathbb{R}^N$. This proves the lemma.
\end{proof}
\begin{proof}[Proof of Corollary \ref{strong-mp}]
Let $u\in  K$ and $u\not\equiv0$. Then,  in the sense of Remark  \ref{rem-3333},
\begin{align*}
(-\Delta + b)(S_b u) = \beta au \geqslant 0
\end{align*}
By Lemma \ref{5.22-22}, either $S_b u>0$ a.e.~or $S_b u = 0$ a.e.
In the second case, we have 
\begin{align*}
(-\Delta + b)(S_b u) = \beta au = 0 \quad \text{a.e.}
\end{align*}
Since $a>0$, we infer that $u=0$ a.e.,  which is a contradiction.
Hence, $S_b u>0$ a.e.
\end{proof}
\begin{proof}[Proof of Lemma \ref{moreau}]
For the proof see \cite[Proposition 6.27]{zbMATH05868403}.
\end{proof}
 \begin{proof}[Proof of Lemma \ref{lem4.8}]
Let us show that for every $w\in K$, there exists $u\in K$,
such that 
 \begin{align}
 \label{uvw-1111}
 \<u,v\rangle^*_{ab} = (w,v)_{L_2} \qquad \forall  v\in E_b.
\end{align}
 One immediately verifies that $u: = (-\Delta  +a(x)S_b + eu)^{-1} w$
satisfies \eqref{uvw-1111}.
By Lemma \ref{big-inverse}, $u\in K$.

Let us show that $K^o \subset - K$. Take, $w\in K^o$,
 $w=w_{+} + w_{-}$, $w_+\in K$. Let 
 \begin{align*}
 u: = (-\Delta  +a(x)S_b + eu)^{-1} w_+.
\end{align*}
 Then, $u\in K$ and
 \begin{align*}
 \|w_{+}\|^2_{L_2} = \<w_{+},w\rangle_{L_2} = \<u, w\rangle_{ab} \leqslant 0.
\end{align*}
 Therefore, $w_+ = 0$ a.e., and hence $w\in -K$.
\end{proof}

 \begin{proof}[Proof of Lemma \ref{LA3}]
 Let $u\in E_b$ and $v\in  K^o\subset- K$.
  By the Moreau decomposition theorem (Lemma \ref{moreau}),
  $u = P_{ K} u + P_{ K^o} u$, 
  where $P_{ K} u  \geqslant 0$ and $P_{ K^o} u \leqslant 0$.
  Since $u\geqslant P_{ K^o} u$ and $v\leqslant 0$ a.e., 
\begin{align*}
    \langle A(u), v\rangle_{ab}^*= 
    \int (f(x, u) +e(x)u) v \, \mathrm{d} x \leqslant 
    \int \left(f(x,  P_{ K^o} u) + e(x) P_{ K^o} u\right) v \, \mathrm{d} x \
    =\left\langle A\left(P_{{K}^o} u\right), v \right\rangle_{ab}^*.
\end{align*}
         Note that since $ K^o \subset - K$, by \cite[Propositions 6.23]{zbMATH05868403} 
         and \cite[Corollary 6.33]{zbMATH05868403}, 
 $(- K)^o \subset  K^{oo} =  K$.
   Let $u\in E_b$ and $w\in  (- K)^o\subset  K$. We have
   $w\geqslant 0$ and $u = P_{- K} u + P_{(- K)^o}u \leqslant P_{(- K)^o}u$ a.e.
  By the same argument as in the previous inequality,   \vspace{2mm}\\
   \phantom{bb} \hspace{45mm} $\langle A(u), w\rangle_{ab}^* \leqslant 
    \left\langle A\left(P_{(-{K})^o} u\right), w\right\rangle_{ab}^*.$
\end{proof}

    \begin{proof}[Proof of Lemma \ref{EDO}]
    Let $w_{1}\in {K}$ and $w_{2}\in -{K}$  be functions with compact supports 
    such that $\|w_{1}\|_{b}=\|w_{2}\|_{b}=1$ and $\overline{{\rm supp}\, w_{1}} \cap\overline{ {\rm supp}\, w_{2} }=\varnothing $. 
    By Proposition \ref{Ne}, there exist a constant $C>0$ such that     
    $ \| t w_{1}+(1-t)w_{2} \|_{ab}\leqslant 2C \|t w_{1}+(1-t)w_{2} \|_{b}=2C$ for all $t\in [0,1]$. By \eqref{ambro}, 
    there exists a constant $C_1>0$ such that
    for all $\lambda\geqslant 1$,
    \begin{align*}
    J\big(\lambda(t w_{1}+(1-t)w_{2})\big) \leqslant \frac{\lambda^2}{2}\| t w_{1}+(1-t)w_{2} \|_{ab}^2 -  \lambda^{\mu_0}  C_1
    \leqslant C \lambda^2  - C_1 \lambda^{\mu_0}.
\end{align*}
    Therefore, there exists $\lambda_0\geqslant 1$ such that 
    $J\big(\lambda_0(t w_{1}+(1-t)w_{2})\big)< 0$ for all $t\in [0,1]$.
    The path $h(t): = \lambda_0(t w_{1}+(1-t)w_{2})$ satisfies  the desirable conditions.
 \end{proof}

\subsection*{Acknowledgements}
The authors acknowledge partial support from CNPq-Brazil 
 (grants 409764/ 2023-0 and 443594/2023-6) and CAPES MATH AMSUD (grant 88887.878894/2023-00). 
 J.~M. do \'O further acknowledges partial support from CNPq through the grant 312340/2021-4 and Para\'iba State Research Foundation (FAPESQ), grant no 3034/2021.
 V\!.V\!.S. additionally acknowledges support from CAPES through the doctoral grant 88887.610117/2021-00 and expresses his gratitude to Prof. Jacques Giacomoni for the insightful discussions, held during V\!.V\!.S's visit to the University of Pau, that helped to improve the quality of the paper.

\subsection*{Statements and Declarations}

\begin{flushleft}
 {\it Ethical Approval:}  Not applicable.\\
 {\it Competing interests:}  Not applicable. \\
 {\it Authors' contributions:}    All authors contributed to the study conception and design. Material preparation, data collection, and analysis were performed by all authors. The authors read and approved the final manuscript.\\
{\it Availability of data and material:}  Not applicable.\\
{\it Ethical Approval:}  All data generated or analyzed during this study are included in this article.\\
{\it Consent to participate:}  All authors consent to participate in this work.\\
{\it Conflict of interest:} The authors declare that they have no conflict of interest. \\
{\it Consent for publication:}  All authors consent for publication. 
\end{flushleft}

\bigskip

\bibliography{ref.bib}
\bibliographystyle{abbrv}

\end{document}